\documentclass[11pt,reqno]{article}
\usepackage{amssymb, amsmath, amsthm, amsfonts, amscd, epsfig, subfig}
\usepackage[dvipsnames]{xcolor}
\usepackage[utf8]{inputenc}
\usepackage[english]{babel}
\usepackage{bm}
\usepackage{bbm}
\usepackage{graphicx, epsfig}
\usepackage{geometry,graphicx,pgfplots}

\usepackage[symbol]{footmisc}
\usepackage{algorithm}
\usepackage{mathtools}

\usepackage{afterpage}

\newtheorem{theorem}{Theorem}[section]
\newtheorem{cor}[theorem]{Corollary}
\newtheorem{lemma}[theorem]{Lemma}                                                                                                                                                                                                                                                                             
\newtheorem{prop}[theorem]{Proposition}

\newtheorem{notation}{Notation}[section]

\newtheorem{example}{Example}
\newtheorem{remark}{Remark}
\newtheorem*{question}{Question}

\def\ep{\varepsilon}

\newcommand{\Om}{\Omega}

\newcommand{\RR}{\mathbb{R}}

\newcommand{\NN}{\mathbb{N}}

\newcommand{\p}{\partial}

\newcommand{\pd}[2]{\frac {\p #1}{\p #2}}
\newcommand{\ds}{\displaystyle}
\newcommand{\eqnref}[1]{(\ref {#1})}

\renewcommand{\qed}{\hfill $\Box$ \medskip}

\newcommand{\beq}{\begin{equation}}
\newcommand{\eeq}{\end{equation}}
\newcommand{\RN}[1]{%
  \textup{\uppercase\expandafter{\romannumeral#1}}%
}

\numberwithin{equation}{section}
\numberwithin{figure}{section}

\begin{document}
\title{Shape monotonicity of the first Steklov--Dirichlet eigenvalue on eccentric annuli
 \thanks{\footnotesize
JH and ML are grateful for support from the Basic Science Research Program through the National Research Foundation of Korea(NRF) funded by the Ministry of Education (NRF-2019R1A6A1A10073887). DS is supported by the National Research Foundation of Korea (NRF) grant number 2017R1D1A1B03036369 funded by the Ministry of Education and grant number 2017R1E1A1A03070543 funded by the Ministry of Science and ICT.}}

\date{}

\author{
Jiho Hong\thanks{Department of Mathematical Sciences, Korea Advanced Institute of Science and Technology, 291 Daehak-ro, Yuseong-gu, Daejeon 34141, Republic of Korea (jihohong@kaist.ac.kr, mklim@kaist.ac.kr).}\and
Mikyoung Lim\footnotemark[2]\ \footnote{Corresponding author.}
\and
Dong-Hwi Seo\thanks{Division of Liberal Arts and Sciences, GIST College, Gwangju Institute of Science and Technology, 123 Cheomdangwagi-ro, Buk-gu, Gwangju 61005, Republic of Korea (donghwi.seo26@gmail.com).}
}

\maketitle

\begin{abstract}
In this paper, we investigate the monotonicity of the first Steklov--Dirichlet eigenvalue on eccentric annuli with respect to the distance, namely $t$, between the centers of the inner and outer boundaries of an annulus. 
 We first show the differentiability of the eigenvalue in $t$ and obtain an integral expression for the derivative value in two and higher dimensions. We then derive an upper bound of the eigenvalue for each $t$, in two dimensions, by the variational formulation. We also obtain a lower bound of the eigenvalue, given a restriction that the two boundaries of the annulus are sufficiently close. The key point of the proof of the lower bound is in analyzing the limit behavior of an infinite series expansion of the first eigenfunction in bipolar coordinates. We also perform numerical experiments that exhibit the monotonicity for two dimensions.\\

\end{abstract}



%
%


\section{Introduction and main results}
We consider the eigenvalue problem for the Laplacian operator on a smooth bounded domain $\Om\subset \RR^n$:
\beq\label{eqn:laplacian}
\Delta u=0\quad \mbox{in }\Om
\eeq
with the Steklov--Dirichlet boundary condition
\begin{align}
\ds u&=0\quad \mbox{on } C_1,\label{BC1}\\
\ds \pd{u}{\nu}&=\sigma u\quad \mbox{on }C_2,\label{BC2}
\end{align}
where $C_1$ and $C_2$ are disjoint components of $\partial \Omega$
and $\nu$ denotes the unit outward normal vector to $\p\Om$.  
If \eqnref{eqn:laplacian}--\eqnref{BC2} admits a non-trivial solution $u$ for a real constant $\sigma$, we call $u$ an eigenfunction and $\sigma$  a Steklov--Dirichlet eigenvalue.  It is well-known (see, for example, \cite{Agranovich:2006:MPS}) that the spectrum of the Steklov--Dirichlet eigenvalue problem is discrete, and the sequence of eigenvalues ordered in an ascending order diverges to infinity. The first (smallest) eigenvalue, namely $\sigma_1(\Om)$, is positive provided that $C_1\neq\phi$.
 When $C_1 = \phi$ and $C_2$ is connected, the equations become the classical Steklov eigenvalue problem (see \cite{stekloff:1902:FPM}; for the historical background, see \cite{Kuzentsov:2014:LVA}). On the other hand, if \eqnref{BC1} is replaced by the zero Neumann condition on $C_1$, the equations is called Steklov--Neumann problem.

The Steklov--Dirichlet and its related eigenvalue problems are of importance from both theoretical and applied perspectives. For example, partially free vibration modes of a thin planar membrane without mass on the interior and with mass on the boundary can be interpreted as Steklov--Dirichlet eigenfunctions \cite{Hersch:1968:EPI}; the Steklov--Neumann problem has been studied in relation to hydrodynamics such as the sloshing problem \cite{Kulczycki:2009:HST}; 
one can construct the so-called free boundary minimal surfaces in a ball by using the concept of the first Steklov eigenvalue \cite{Fan:2015:EPS,Fraser:2020:ECS,Fraser:2013:MSE,Fraser:2016:SEB,Matthiesen:2020:FBM,Petrides:2019:MSE}. The Steklov eigenvalue problem is also related to the classical Laplacian eigenvalue problem $\Delta u = \sigma u$. For instance, the Steklov eigenvalue can be regarded as limiting Neumann eigenvalues of the Laplacian \cite{Arrieta:2008:FRL, Lamberti:2015:VSE, Lamberti:2017:NSA}. It is worth mentioning that the Laplace eigenvalue problems with the Dirichlet, Neumann and Robin boundary conditions have been intensively studied; we refer the reader to a review article \cite{Grebenkov:2013:GSL} and references therein for the properties of the Laplacian eigenvalues and eigenfunctions in various aspects.

Much attention has been focused on the geometric dependence of the first Steklov--Dirichlet eigenvalue in the study of the Steklov--Dirichlet eigenvalue problem. 
For a planar domain, the upper and lower bounds for the first eigenvalue have been studied by applying the variational approach and the conformal mapping technique \cite{Dittmar:1998:IIS,Dittmar:2000:ZKE,Dittmar:2003:MSE,Hersch:1968:EPI}.
In particular, Hersch and Payne obtained an upper bound for planar annular domains in 1968 \cite{Hersch:1968:EPI}. 
Later, Dittmar and Solynin obtained a lower bound for planar annular domains under some geometric restrictions \cite{Dittmar:2003:MSE}. We refer the reader to the survey paper by Dittmar \cite{Dittmar:2005:EPC} for more details. 
For higher dimensions, Ba\~{n}uelos et al. obtained a domain monotonicity result and found an inequality relation between the Steklov--Dirichlet and Steklov--Neumann eigenvalues \cite{Banuelos:2010:EIM}.
Recently, Santhanam and Verma considered the Steklov--Dirichlet eigenvalue on eccentric annuli in $\mathbb{R}^n$,  $n > 2$, with the zero Dirichlet condition on the inner boundary, and showed that the first eigenvalue attains the maximum when the annulus is concentric \cite{Verma:2018:EPL}. Seo extended this maximality to some two-point homogeneous spaces including $\mathbb{R}^2$ \cite{Seo:2019:SOP} (see also the work by Ftouhi \cite{Ftouhi:2019:WPS}). Furthermore, Paoli, Piscitelli, and Sannipoli obtained a stability result on this Steklov--Dirichlet eigenvalue problem \cite{Paoli:2020:SRS}.

In the present paper, we investigate the monotonicity of the first Steklov--Dirichlet eigenvalue on eccentric annuli with respect to the distance between the centers of the inner and outer boundaries of an annulus. 
To state our problem and main results more precisely, we introduce some notations.
Let $B^t_1$ and $B_2$ be the two balls in $\mathbb{R}^n$ with $n\geq 2$ given by
\beq\label{eqn:B1B2}
B_1^t=B(te_1,r_1),\quad B_2=B(0,r_2),\quad 0<r_1<r_2,\quad 0\leq t<r_2-r_1,
\eeq
where $B(x,r)$ denotes the ball centered at $x\in\RR^n$ with radius $r$, and $e_1$ is the unit vector $(1,0,\cdots,0)\in\RR^n$.
We set $B_1=B_1^0$. Note that $\overline{B_1^t} \subset B_2$ for all $t$ in $[0, r_2-r_1)$ and that $B_1^t$ is concentric with $B_2$ at $t=0$; Figure \ref{fig:geometry} illustrates $B_1$, $B_1^t$ and $B_2$.
We impose zero Dirichlet condition on $\p B_1^t$ and Robin condition on $\p B_2$. In other words, we consider the problem 
\begin{align} \label{problem}
\begin{cases}
    \ds \Delta u^t = 0   &\ds\text{in   } B_2\setminus \overline{B_1^t},\\
  \ds   u^t = 0 & \ds\text{on   } \partial B_1^t,\\
  \ds   \frac{\partial u^t}{\partial \nu} = \sigma^t u^t &\ds\text{on   } \partial B_2.
\end{cases}
\end{align}
Throughout this paper, we add the superscript $t$ to $u$ and $\sigma$ in order to indicate their dependence on the parameter $t$. 
The first eigenvalue, $\sigma_1^t$, attains the maximum at $t=0$, the concentric case, as shown in \cite{Ftouhi:2019:WPS, Verma:2018:EPL, Seo:2019:SOP}; the maximal value is $\left(r_2(\ln r_2 - \ln r_1)\right)^{-1}$ in two dimensions.
 Beyond this, we ask the following:
\begin{question}
Is $\sigma_1^t$ monotone decreasing as $t$ increases?  
\end{question}
\begin{figure}[h!]
    \centering
    \begin{tikzpicture}

\coordinate  (X) at (0,0);
\coordinate (D) at (1.732,1);

\fill[gray!40,even odd rule] (X) circle (0.7) (X) circle (2);	

\draw (X) circle(2cm);
\draw (X) circle (0.7cm);
\draw (X) --(D);
\draw (X) --(-0.7,0);

\node[above, scale=0.8] at (0.866,0.5) {$r_2$};
\node[above, scale=0.8] at (-0.35,0) {$r_1$};
\node at (0,2.25){$B_2$};
\node at (0,0.95){$B_1$}; 
\node[below] at (0,0){$O$};
\foreach \point in {X}
	\fill [black] (\point) circle (1.5pt);

	\end{tikzpicture}
	\hskip 1.5cm
        \begin{tikzpicture}
\coordinate  (C) at (0,0);
\coordinate (D) at (2,0);
\coordinate (X) at (1,0);
\fill[gray!40,even odd rule] (X) circle (0.7) (C) circle (2);
\draw (C) circle(2cm);
\draw (X) circle (0.7cm);

\draw (X) --(C);
\node at (0.5,0.2) {$t$};

\node at (1,0.95){$B_1^t$};
\node at (0,2.25){$B_2$};
\node[below] at (0,0){$O$};

\foreach \point in {X,C}
	\fill [black] (\point) circle (1.5pt);

	\end{tikzpicture}
    \caption{Problem geometry: an eccentric annulus $\Om=B_2\setminus\overline{B_1^t}$ (the gray region in the right figure). 
    The two balls are concentric when $t=0$, as in the left figure.}
    \label{fig:geometry}
\end{figure}
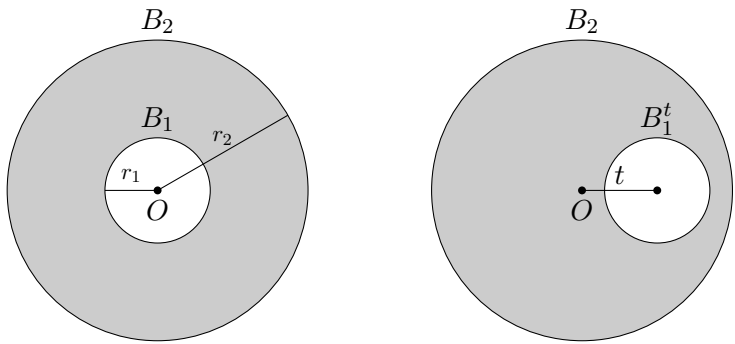

For various Laplacian eigenvalue problems with the vanishing boundary condition, the shape monotonicity of the first eigenvalue on eccentric annuli was verified in previous literature and, in the meantime, similar formulas to \eqnref{int:shape:deri} were derived, for example, by Ramm and Shivakumar \cite{Ramm:1998:IME}, Kesavan \cite{Kesavan:2003:TFL}, Aithal and Anisa \cite{Aithal:2005:TFL}, Anisa and Vemuri \cite{Anisa:2013:TFL}, Anisa and Mahadevan \cite{Anisa:2015:EOP}, Anoop et al. \cite{Anoop:2018:SMF}, and Aithal and Rane \cite{aithal:2019:FDE}.
An essential property of the first eigenvalue to prove its monotonicity in the literature is that 
the first eigenfunction does not have a sign change in the interior but vanishes on the boundary of an annulus.
Besides this property, standard techniques for the Laplacian operator, such as the reflection principle and the maximum principle, were used to show that the derivative in $t$ of the first eigenvalue is negative.

For the eigenvalue problems with the Steklov boundary condition, there have been results which 
derive an integral formula of the first eigenvalue or verify the maximality at the concentric annulus. 
In \cite{Dambrine:2016:EEP}, Dambrine et al. obtained an integral formula for the shape derivative of the first eigenvalue of the Wentzell--Laplacian problem, which is a generalization of the Steklov eigenvalue problem. Rodriguez-Qui\~{n}ones then considered the Steklov eigenvalue problem in two dimensions, based on the result in \cite{Dambrine:2016:EEP}, and showed that the concentric annulus is a critical domain among a class of doubly connected domains \cite{Quinones:2019:CDF}.
It is then proved by Ftouhi \cite{Ftouhi:2019:WPS} that the first Steklov eigenvalue on eccentric annuli attains the maximum at the concentric case. 
For the Steklov--Dirichlet eigenvalue problem, as stated previously, the maximality at the concentric annulus was shown by Santhanam and Verma \cite{Verma:2018:EPL}, Seo \cite{Seo:2019:SOP}, and Ftouhi \cite{Ftouhi:2019:WPS}.
However, to the best of our knowledge, the monotonicity is not known either for the Steklov or for the Steklov--Dirichlet eigenvalue problems on eccentric annuli.
 It is challenging to verify the monotonicity for the eigenvalue with the Steklov boundary condition because the maximum principle technique cannot be applied, which is different from the Laplacian eigenvalue problems.

In this paper, first, we show the differentiability of the first eigenvalue and its associated eigenfunction and derive an integral representation for the derivative in $t$ of the first eigenvalue; see Theorems \ref{thm:diff}, \ref{thm:variation} below (the proofs are provided in section \ref{sec:proof:main:theorems}). It is worth remarking that for the Laplacian eigenvalue problem, an integral formula similar to \eqnref{int:shape:deri} was essentially used to prove the monotonicity of the first eigenvalue. 
\begin{theorem}[Differentiability]\label{thm:diff}
Let $B^t_1$ and $B_2$ be the two eccentric balls in $\mathbb{R}^n$ ($n\ge 2$) given by \eqnref{eqn:B1B2}. Let $\sigma_1^t$ and $u_1^t$ be the first eigenvalue of the Steklov--Dirichlet eigenvalue problem \eqnref{problem} and the corresponding eigenfunction, respectively. Then, the functions  
$$    t \mapsto \sigma_1^t \in \mathbb{R}\quad\mbox{and}\quad  t \mapsto u_1^t \in H^1(B_2)$$ are differentiable with respect to $t$ in $[0,r_2-r_1)$. Here, $H^1$ denotes the Sobolev space of order $1$.
\end{theorem}
\begin{theorem}[Shape derivative]\label{thm:variation} Let $B_1^t$, $B_2$, $\sigma_1^t$ and $u_1^t$ be given as in Theorem \ref{thm:diff}. The shape derivative of the first eigenvalue $\sigma_1^t$ can be expressed as an integral in terms of its associated normalized eigenfunction: for all $t\in[0,r_2-r_1)$,
\beq\label{int:shape:deri}
   \frac{d}{dt}\,\sigma_1^t = -\int_{\partial B_1^t} \left(\frac{\partial u_1^t}{\partial \nu}\right)^2 ( \nu \cdot e_1)\, dS.
\eeq
\end{theorem}

For further analysis in two dimensions, we employ bipolar coordinates, $(\xi,\theta)\in \RR\times(-\pi,\pi]$, satisfying that $\p B_1^t$ and $\p B_2$ are $\xi$-level curves of the values, namely, $\xi_1$ and $\xi_2$, respectively (for details, see section \ref{ch:BipCoord}). The first eigenfunction is naturally expanded into a series of Fourier modes in bipolar coordinates and so is the derivative of the first eigenvalue from the integral formula \eqnref{int:shape:deri}. It is then helpful to investigate the behavior of the coefficients of the series expansions to understand the properties of the eigenvalue and eigenfunction. 
In particular, we deduce the behavior of the ratio of consecutive coefficients in the series expansion of the derivative of the first eigenfunction, namely $F_n$ (see \eqnref{def:Fn} and \eqnref{seriesnormalfirst}), as follows. The proof is in subsection \ref{sec:Proof:Thm1.3}.
\begin{theorem}\label{cor:coefficient relations}
Let $\Om$ be an eccentric annulus in two dimensions. We have 
$$\lim_{n\rightarrow\infty}F_n=-e^{-\xi_2}.$$
\end{theorem}

Dittmar and Solynin obtained a lower bound of the first Steklov--Dirichlet eigenvalue on a ring domain in two dimensions \cite{Dittmar:2003:MSE}.
In the following theorem, by using Theorem \ref{cor:coefficient relations}, we derive a lower bound for the liminf of the first eigenvalue, which provides a finer estimate under some conditions on $r_1,r_2$ and $t$; see Remark \ref{remark:DS} in subsection \ref{sec:asymp:lowerbound} for details. The proof of Theorem  \ref{prop:asymp:lowerbound} is provided in subsection \ref{sec:asymp:lowerbound}.

\begin{theorem}[Lower bound in two dimensions]\label{prop:asymp:lowerbound} 
Let $\sigma_1^t$ be the first Steklov--Dirichlet eigenvalue of $\Om=B_2\setminus\overline{B_1^t}$ in two dimensions. 
It holds that
\beq\label{lower:liminf}
\liminf_{t\to (r_2-r_1)^-}\sigma_1^t\ge\frac{\ds r_1}{\ds2r_2(r_2-r_1)}.
\eeq
\end{theorem}

In addition to analytical results, we perform numerical computation in two dimensions by using bipolar coordinates. Numerical results show that $\sigma_1^t$ is monotone decreasing as $t$ increases.

The first Steklov--Dirichlet eigenvalue on eccentric annuli in three dimensions will be discussed by using bispherical coordinates in a separate paper.

The remainder of this paper is organized as follows. In section 2, we provide the variational characterization for the first Steklov--Dirichlet eigenvalue. Section 3 is devoted to deriving the integral formula for the shape derivative of the first eigenvalue, after showing its differentiability. In section 4, by using the bipolar coordinates, we investigate analytic properties of the first eigenvalue and eigenfunction in two dimensions. We then further derive their asymptotic behaviors in section 5. In section 6, we provide numerical evidence for the monotonicity of the first eigenvalue for two dimensions. We finish with the conclusion in section 7.

\section{Simplicity and upper bound of the first eigenvalue}\label{ch:2}

The first eigenvalue $\sigma_1(\Omega)$ to the Steklov--Dirichlet problem \eqnref{eqn:laplacian}--\eqnref{BC2} admits the variational characterization (see, for example, \cite{Bandle:1980:IIA}):
\begin{align} \label{variational characterization}
    \sigma_1(\Omega) = \inf \left\{ {\ds\int_{\Omega}\left|\nabla v\right|^2 dx} \,\Big|\, v \in H^1(\Omega)\setminus\{0\},\ v=0 \textnormal{ on } C_1,\ \mbox{and }{\ds\int_{C_2}v^2 \,dS=1}  \right\}.
\end{align}
Recall that 
$
    \sigma_1^t = \sigma_1(B_2\setminus \overline{B_1^t})
$
with $C_1=\p B_1^t$ and $C_2=\p B_2$.
It admits the normalized eigenfunction $u_1^t \in H^1(B_2\setminus \overline{B_1^t})$ satisfying
\begin{align}
	&u_1^t\ge 0 \quad\text{in } B_2\setminus \overline{B_1^t},\label{eqn:normalizedpositive} \\
    \int_{\partial B_2}&\left(u_1^t\right)^2 \,dS =1.\label{eqn:normalized}
\end{align}
The positiveness (\ref{eqn:normalizedpositive}) is supported by the following lemma.
\begin{lemma}\label{positive}
The first eigenfunction $u_1^t$ does not change the sign in $B_2\setminus \overline{B_1^t}$, and  $\sigma_1^t$ is simple.
\end{lemma}
\begin{proof}
    Suppose $u_1^t$ has both positive and negative values. Let $(u_1^t)^+ = \text{max}(u_1^t,0)$ and $(u_1^t)^- =\text{max}(-u_1^t,0)$.  Then, from the smoothness of $u_1^t$, we have  $(u_1^t)^+, (u_1^t)^- \in H^1(\Omega) \setminus \{0\}$ and
    \begin{align*}
      \ds  \int_{B_2\setminus \overline{B_1^t}}\left|\nabla u_1^t\right|^2 dx &= \ds\int_{B_2\setminus \overline{B_1^t}} \left|\nabla (u_1^t)^+\right|^2dx + \int_{B_2 \setminus \overline{B_1^t}} \left|\nabla (u_1^t)^-\right|^2 dx, \\
        \int_{\p{B_2}} \left(u_1^t\right)^2 dS &= \int_{\p{B_2}} \left((u_1^t)^+\right)^2 dS + \int_{\p{B_2}} \left((u_1^t)^-\right)^2 dS.
    \end{align*}
    Since $u_1^t$ is the first eigenfunction, the variational characterization (\ref{variational characterization}) implies
    \begin{align*}
        \sigma_1^t &= \frac{\ds\int_{B_2\setminus \overline{B_1^t}}\left|\nabla u_1^t\right|^2dx}{\ds\int_{\partial B_2}\left(u_1^t\right)^2 dS} 
        \ge \text{min} \left( \frac{\ds\int_{B_2\setminus \overline{B_1^t}} \left|\nabla (u_1^t)^+\right|^2dx}{\ds\int_{\p{B_2}} \left((u_1^t)^+\right)^2 dS}, \frac{\ds\int_{B_2 \setminus \overline{B_1^t}} \left|\nabla (u_1^t)^-\right|^2 dx}{\ds\int_{\p{B_2}} \left((u_1^t)^-\right)^2 dS} \right)
        \ge \sigma_1^t.
    \end{align*}
Therefore, $(u_1^t)^+$ or $(u_1^t)^-$ is also the first eigenfunction from the variational characterization (\ref{variational characterization}) and, thus, satisfies \eqnref{problem}. Both $(u_1^t)^+$ and $(u_1^t)^-$ are positive in some open subset of $B_2\setminus \overline{B_1^t}$ from the assumption on $u_1^t$ and, consequently, they are zero in some open subset. This contradicts, in view of the maximum principle, that $(u_1^t)^+$ or $(u_1^t)^-$ satisfies \eqnref{problem}.
 Therefore the first eigenfunction $u_1^t$ is not sign-changing. Then a function orthogonal to $u_1^t$ is sign-changing or constantly zero, so it cannot be the first eigenfunction. It implies that $\sigma_1^t$ is simple.
\end{proof}

From \eqnref{variational characterization}, we obtain an upper bound of $\sigma_1^t$ (it will be used in subsection \ref{sec:Proof:Thm1.3}):
 \begin{theorem}[Upper bound in two dimensions]\label{upperbound:global} 
 Let $B^t_1$ and $B_2$ be the two eccentric balls in $\mathbb{R}^2$ given by \eqnref{eqn:B1B2}. For any $t\in[0,r_2-r_1)$, it holds that 
\beq\label{eqn:upperbound:global:int}
\sigma_1^t\le \frac{\ds \pi( r_2^2-r_1^2)}{\ds 2\pi r_2(r_2^2+r_1^2+t^2) - 4r_1r_2\int_0^\pi\sqrt{r_2^2-2r_2 t\cos\varphi+t^2}\,d\varphi}.
\eeq
\end{theorem}
\begin{proof}
We take a test function $v(x):=\big|x-(t,0)\big|-r_1$. Note that $v\big|_{\p B_1}=0$ and $v\big|_{\p B_2}=\sqrt{r_2^2 - 2r_2t\cos\varphi +t^2}-r_1$ by parametrizing $\p B_2$ as $(r_2\cos\varphi, r_2\sin\varphi)$, $-\pi\leq\varphi<\pi$.
It then holds that
\begin{equation}\label{UB:numerator}
\int_{B_2\setminus \overline{B_1^t}} |\nabla v|^2\,dx
=\int_{B_2\setminus \overline{B_1^t}} 1\,dx
=\pi(r_2^2-r_1^2)
\end{equation}
and
\begin{align}
\notag\ds \int_{\partial B_2}v^2\,dS
=2\pi r_2(r_2^2+r_1^2+t^2)-4r_1r_2\int_{0}^{\pi}\sqrt{r_2^2-2r_2t\cos\varphi+t^2}\,d\varphi.
\end{align}
From \eqnref{variational characterization}, we prove the theorem.
\end{proof}

\section{Differentiability of the first eigenvalue and its shape derivative}\label{sec:proof:main:theorems}

We can identify the two function spaces 
\begin{align}\label{two_spaces}
   \left \{ u \in H^1(B_2\setminus{\overline{B_1^t}}) \,\Big|\, u=0 \text{ on } \partial B_1^t\right\}
   \quad\mbox{and}\quad
   H^1_{\bar{B}_1^t}(B_2),
\end{align}
where
$H^1_A(\Omega) := \left\{ u\in H^1(\Omega) \,\big|\, u=0 \text{ in } A\right\}
$
for a subset $A\subset\overline{\Om}$.
Hence, we can regard $u_1^t$ as a function in $H^1_{\bar{B}_1^t}(B_2)\subset H^1(B_2)$, and $\sigma_1^t$ admits also the following variational characterization:
\begin{align}\label{variational characterization2}
    \sigma_1^t = \inf \left\{ \left.\frac{\ds\int_{B_2}|\nabla v|^2 \,dx}{\ds\int_{\partial B_2}v^2 \,dS} \right| v \in H^1_{\bar{B}_1^t}(B_2)\setminus\{0\}\right\}.
\end{align}
In the remaining of the section we prove Theorems \ref{thm:diff}, \ref{thm:variation} by using \eqnref{variational characterization2}.

\subsection{Proof of Theorem \ref{thm:diff} (differentiability of $\sigma_1^t$ and $u_1^t$)}
The outline of the proof follows an argument of \cite{Bonder:2007:OFS}.

For $t_0\in[0,r_2-r_1)$ and $s>0$ satisfying $t_0+s<r_2-r_1$, we consider the first eigenfunction $\sigma_1^{t_0+s}$ and its associated normalized eigenfunction $u_1^{t_0+s}$ of the eigenvalue problem \eqnref{problem} (see also \eqnref{eqn:normalizedpositive} and \eqnref{eqn:normalized}). Then, the following weak formulation holds using the function spaces identification \eqnref{two_spaces}:
 \begin{align}\label{weak formulation}
    \int_{B_2} \nabla u_1^{t_0+s}\cdot\nabla \varphi\, dx = \sigma_1^{t_0+s} \int_{\partial B_2} u_1^{t_0+s}\varphi\, dS\quad\mbox{for all }\varphi \in H_{\bar{B}_1^{t_0+s}}^1(B_2).
\end{align}

Let $V: \overline{B_2} \rightarrow \mathbb{R}^n$ be a variation field on $\overline{B_2}$ generated by the moving of ${B_1^{t_0}}$ to $e_1$-direction fixing $\partial B_2$. In particular, $V$ is a smooth vector field satisfying
\begin{align*}
    V=e_1 \text{ on } \overline{B_1^{t_0}} \quad\mbox{and}\quad \mbox{supp}(V) \subset B_2.
\end{align*}
We now define a map $\Phi: (-r_2+r_1-t_0,\,r_2-r_1-t_0)\times B_2 \rightarrow \mathbb{R}^n$ by
\begin{align*}
    \Phi(s,x) = x+sV(x).
\end{align*}
Clearly, it holds that $\Phi(s, B_2)=B_2$ and $u_1^{t_0+s}\circ \Phi(s,\boldsymbol{\cdot})\in H^1_{\bar{B}_1^{t_0}}$. Since 
\beq\label{Dphi:0}
D\Phi(0,\boldsymbol{\cdot}) = Id,
\eeq
 there is a neighborhood $U_1$ of $0$ in $\mathbb{R}$ such that $\Phi(s,\boldsymbol{\cdot})$ is a diffeomorphism of $B_2$. 
By the change of variables formula and the chain rule, (\ref{weak formulation}) becomes
\begin{align}\notag
   & \int_{B_2} \Big(\nabla \left(u_1^{t_0+s}\circ\Phi(s,\boldsymbol{\cdot})\right)\left(D\Phi(s,\boldsymbol{\cdot})\right)^{-1}\Big) \cdot \Big(\nabla \left(\varphi \circ \Phi(s,\boldsymbol{\cdot})\right)\left(D\Phi(s,\boldsymbol{\cdot})\right)^{-1}\Big) \left| D\Phi(s, \boldsymbol{\cdot}) \right| dx \\
  =&\, \sigma_1^{t_0+s}\int_{\partial B_2} \left(u_1^{t_0+s}\circ \Phi(s,\boldsymbol{\cdot})\right)\left(\varphi \circ \Phi(s,\boldsymbol{\cdot})\right)dS,\label{eqn:f_1}
\end{align}
and (\ref{eqn:normalized}) becomes
\begin{align}\label{eqn:f_2}
     \int_{\partial B_2} \left(u_1^{t_0+s}\circ \Phi(s,\boldsymbol{\cdot})\right)^2 dS =1.
\end{align}

We denote $\big(H_{\bar{B}_1^{t_0}}^1(B_2)\big)'$ the dual space of $H_{\bar{B}_1^{t_0}}^1(B_2)$ and $\left\langle \boldsymbol{\cdot}, \boldsymbol{\cdot} \right\rangle$ the dual pairing between $\big(H^1_{\bar{B}_1^{t_0}}(B_2)\big)'$ and $H^1_{\bar{B}_1^{t_0}}(B_2)$. 
We then define
$$
    f=(f_1,f_2):\, U_1 \times H_{\bar{B}_1^{t_0}}^1(B_2) \times \mathbb{R} \rightarrow \big(H_{\bar{B}_1^{t_0}}^1(B_2)\big)' \times \mathbb{R}$$
by
\begin{align*}
\begin{cases}
    \ds \left\langle f_1(s, v, \sigma),\, \psi \right\rangle &=  \ds\int_{B_2} \big((\nabla v)\left(D\Phi(s,\boldsymbol{\cdot})\right)^{-1}\big) \cdot \big((\nabla \psi)(D\Phi(s,\boldsymbol{\cdot}))^{-1}\big) \left| D\Phi(s, \boldsymbol{\cdot}) \right| dx - \sigma \int_{\partial B_2} v\psi\,dS,\\[3mm]
    f_2(s, v, \sigma) &= \ds \int_{\partial B_2} v^2 \,dS -1
\end{cases}
\end{align*}
for all $\psi \in H_{\bar{B}_1^{t_0}}^1(B_2)$. 
Clearly, $f$ is $C^1$ near $(0,u_1^{t_0},\sigma_1^{t_0})$. In addition, we set
\begin{align*}
    g :&\, U_1\rightarrow H_{\bar{B}_1^{t_0}}^1(B_2) \times \mathbb{R} \\
   & s\longmapsto\left(u_1^{t_0+s}\circ \Phi(s,\boldsymbol{\cdot}),\  \sigma_1^{t_0+s}\right).
\end{align*}
Then, equations (\ref{eqn:f_1}) and (\ref{eqn:f_2}) imply $$f(s,g(s))=0\quad\mbox{for all }s\in U_1.$$ If we show that  
\begin{align}\label{isomorphism}
    \frac{\partial f}{\partial (v, \sigma)}\Big|_{(0, u_1^{t_0},\sigma_1^{t_0})} : H_{\bar{B}_1^{t_0}}^1(B_2) \times \mathbb{R} \longrightarrow \big(H_{\bar{B}_1^{t_0}}^1(B_2)\big)' \times \mathbb{R}
\end{align}
is an isomorphism, then $g$ is $C^1$ by the implicit function theorem (see, for instance, \cite[Theorem 4.B]{Zeidler:1986:NFA}) and Theorem \ref{thm:diff} is proved.

In the remaining portion of the proof, we show that (\ref{isomorphism}) is an isomorphism. 
That is, for any $(h, \lambda) \in \big(H^1_{\bar{B}_1^{t_0}}(B_2)\big)'\times \mathbb{R}$, we will find a unique element $(w, \mu) \in H^1_{\bar{B}_1^{t_0}}(B_2) \times \mathbb{R}$ such that 
\begin{align} \label{fredholm1}
   \left\langle h, \psi \right\rangle 
&=\left\langle  \frac{\partial f_1}{\partial (v, \sigma)}\Big|_{(0, u_1^{t_0},\sigma_1^{t_0})}(w, \mu), \psi\right\rangle\quad\mbox{for all }\psi\in  H^1_{\bar{B}_1^{t_0}}(B_2) ,\\
    \label{fredholm2}
   \lambda&=\frac{\partial f_2}{\partial (v,\sigma)}\Big|_{(0, u_1^{t_0},\sigma_1^{t_0})}(w,\mu),
\end{align}
where, from \eqnref{Dphi:0} and the definition of $f$, the right-hand sides are
\begin{align}\label{df1:expan}
  &\left\langle  \frac{\partial f_1}{\partial (v, \sigma)}\Big|_{(0, u_1^{t_0},\sigma_1^{t_0})}(w, \mu), \psi\right\rangle = \int_{B_2} \nabla w \cdot\nabla \psi\,dx - \int_{\partial B_2}\left(\sigma_1^{t_0}w+\mu u_1^{t_0}\right)\psi \,dS, \\
&\frac{\partial f_2}{\partial (v,\sigma)}\Big|_{(0, u_1^{t_0},\sigma_1^{t_0})}(w,\mu) = 2\int_{\partial B_2}u_1^{t_0}w \,dS.\label{df2:expan}
\end{align}

We define two linear maps
$
    S_1, S_2 : H_{\bar{B}_1^{t_0}}^1(B_2) \longrightarrow \big(H_{\bar{B}_1^{t_0}}^1(B_2)\big)' 
$
by
\begin{align*}
    \left\langle S_1(v), \psi\right\rangle &= \int_{\partial B_2} v \psi \,dS, \\
    \left\langle S_2(v), \psi\right\rangle &= \int_{B_2} v \psi \,dx\qquad\mbox{for }v,\psi\in H_{\bar{B}_1^{t_0}}^1(B_2).
\end{align*}
In fact, $S_1$ is the composition of the following three mappings:
\begin{align*}
H_{\bar{B}_1^{t_0}}^1(B_2) \subset\subset L^2(\partial B_2) \overset{isometry}{\longrightarrow} (L^2(\partial B_2))'\subset\subset (H_{\bar{B}_1^{t_0}}^1(B_2))',
\end{align*}
where $(L^2(\partial B_2))'$ is the dual space of $L^2(\partial B_2)$. The first inclusion is compact from the compact embedding $H^1 (B_2) \subset\subset L^2(\partial B_2)$ (see, e.g., \cite[Theorem 2.31, 2.33]{Labrie:2017:TSS} or \cite[Theorem 1.2 in Chapter 1]{Necas:2012:DMT}). Furthermore, the Schauder theorem (see, e.g., \cite[Theorem 6.4]{Brezis:2011:FAS}) implies that the third map is compact and, thus, $S_1$ is compact. Similarly, $S_2$ is compact.

Now, we go back to equation (\ref{fredholm1}), which, in view of \eqnref{df1:expan}, can be rewritten as
\begin{align}\notag
    \left\langle h,\psi \right\rangle = (w, \psi) -\left\langle (\sigma_1^{t_0}S_1+S_2)(w), \psi\right\rangle-\left\langle \mu S_1(u_1^{t_0}), \psi \right\rangle,
\end{align}
 where $(w, \psi)$ is the inner product in $H^1_{B_1^{t_0}}(B_2)$. In other words,
 \begin{align}\label{eqn:for:h}
   (w, \psi) -\left\langle (\sigma_1^{t_0}S_1+S_2)(w), \psi\right\rangle  =  \left\langle h,\psi\right\rangle+\left\langle \mu S_1(u_1^{t_0}), \psi \right\rangle.
\end{align}
We can regard $h$, $(\sigma_1^{t_0}S_1+S_2)(w)$ and $S_1(u_1^{t_0})$, that are elements in the dual space of $H^1_{\bar{B}_1^{t_0}}(B_2)$, as functions in $H^1_{\bar{B}_1^{t_0}}(B_2)$. Since $\sigma_1^{t_0}S_1+S_2$ is compact and self-adjoint, there is a unique solution $w$ to \eqnref{eqn:for:h} up to $\text{kernel}\left(Id-(\sigma_1^{t_0}S_1+S_2)\right)= \text{span}(u_1^{t_0})$ if and only if 
\begin{align} \label{perp}
    h + \mu S_1(u_1^{t_0}) \perp u_1^{t_0},
\end{align}
by the Fredholm alternative (see, e.g., \cite[Corollary 8.1]{SBH:2019:VTE}).
Because of $
\left\langle S_1(u_1^{t_0}), u_1^{t_0}  \right\rangle \neq 0,    
$
we can uniquely find $\mu$ satisfying (\ref{perp}). Furthermore, equation \eqnref{eqn:for:h} with \eqnref{fredholm2} and \eqnref{df2:expan} uniquely determine $w$. Therefore, (\ref{isomorphism}) is an isomorphism and it finishes the proof. 
\qed

\subsection{Proof of Theorem \ref{thm:variation} (shape derivative of $\sigma_1^t$)}
   We use the same notation of $V, \Phi$ as in the proof of Theorem \ref{thm:diff}.
   
For $x\in \partial B_1^{t_0}$, we have
    \begin{align*}
        0=\left.\frac{d}{ds}\right|_{s=0}\big(u_1^{t_0+s}(\Phi(s,x))\big) &= (u_1^{t_0})'(x) +( \nabla u_1^{t_0} \cdot e_1 ) 
        = (u_1^{t_0})'(x)+ \frac{\partial u_1^{t_0}}{\partial \nu^{t_0}}\, ( \nu^{t_0} \cdot e_1 ).
    \end{align*}
Hence, it follows that
    \begin{align}\notag
\begin{cases}
  \ds   \Delta (u_1^{t_0})' =0  &\text{in   } B_2\setminus \overline{B_1^{t_0}},   \\
    \ds (u_1^{t_0})' = -\frac{\partial u_1^{t_0}}{\partial \nu^{t_0}}\,( \nu^{t_0} \cdot e_1 ) &\text{on   } \partial B_1^{t_0}.
\end{cases}
\end{align}
We then obtain by using the Robin boundary condition on $\p B_2$ and (\ref{eqn:normalized}) that
\begin{align}
\int_{\partial (B_2\setminus \overline{B_1^{t_0}}) } \frac{\partial u_1^{t_0}}{\partial \nu^{t_0}}\, (u_1^{t_0})'  \,dS 
=&\int_{\partial B_1^{t_0} } \frac{\partial u_1^{t_0}}{\partial \nu^{t_0}} \,(u_1^{t_0})'  \,dS 
+\sigma_1^{t_0} \int_{\partial B_2 } u_1^{t_0} (u_1^{t_0})'  \,dS \notag\\ \label{u:int:onB1}
=& -\int_{\partial B_1^{t_0}} \left(\frac{\partial u_1^{t_0}}{\partial \nu^{t_0}}\right)^2(\nu^{t_0} \cdot e_1) \,dS. 
\end{align}

On the other hand, for $y\in \partial B_2$, we have
\begin{align*}
    \frac{\partial u_1^{t_0+s}}{\partial \nu} (y) = \sigma_1^{t_0+s}u_1^{t_0+s}(y) \quad\mbox{and}\quad \Phi(s,y)=y,
\end{align*}
which implies
\begin{align*}
    \frac{\partial (u_1^{t_0})'}{\partial \nu} (y) = (\sigma_1^{t_0})'u_1^{t_0}(y)+\sigma_1^{t_0}(u_1^{t_0})'(y).
\end{align*}
Then, from the vanishing boundary condition on $\p B_1$ and (\ref{eqn:normalized}), we arrive at
\begin{align} \label{derivative}
     \int_{\partial (B_2\setminus \overline{B_1^{t_0}})}u_1^{t_0}\,\frac{\partial (u_1^{t_0})'}{\partial \nu^{t_0 }}  \,dS = \int_{\partial B_2} u_1^{t_0} \, \frac{\partial (u_1^{t_0})'}{\partial \nu^{t_0 }}\,dS=(\sigma_1^{t_0})'.
\end{align}
Using equations \eqnref{u:int:onB1}, (\ref{derivative}) and Green's identity 
\begin{align*}
    \int_{\partial (B_2\setminus \overline{B_1^{t_0}})} u_1^{t_0}\,\frac{\partial (u_1^{t_0})'}{\partial \nu^{t_0}} \,dS 
    &=\int_{\partial (B_2\setminus \overline{B_1^{t_0}}) } \frac{\partial u_1^{t_0}}{\partial \nu^{t_0}} \,(u_1^{t_0})'  \,dS,
    \end{align*}
we obtain the desired identity (\ref{int:shape:deri}).
\qed

\section{The first Steklov--Dirichlet eigenfunction in two dimensions in bipolar coordinates}\label{ch:BipCoord}
In sections \ref{ch:BipCoord} and \ref{sec:asymp}, we deal with the first Steklov--Dirichlet eigenfunction on eccentric annuli in two dimensions. 
We use the bipolar coordinate system because of its convenience in solving the Laplace problem subject to boundary conditions on two circular interfaces.
 It is worth mentioning that the electric field concentration in composite materials has been successfully analyzed using bipolar or bispherical coordinates in \cite{Ammari:2005:GES, Kang:2014:CEF, Kim:2018:EFC, Lim:2015:ASC}.

Later, in section \ref{sec:numerical}, we numerically observe the monotonicity of $\sigma_1^t$ in $t$ by using the series expansion of the first eigenfunction in bipolar coordinates.

\begin{figure}[h!]
\begin{center}

\begin{tikzpicture}[scale=0.43]

\coordinate  (C) at (3.537742925, 0);
\coordinate (X) at (2.125, 0);
\fill[gray!40,even odd rule] (X) circle (1) (C) circle (3);

\node at (1.7,1.37){$B_1^t$};
\node at (3.537742925,3.35){$B_2$};

\draw[thick] (3.537742925, 0) circle (3);
\draw[thick] (2.125, 0) circle (1);
\draw[thick] (-3.537742925, 0) circle (3);
\draw[thick] (-2.125, 0) circle (1);

\draw (0, -6.0) -- (0, 6.0);
\draw (-9.27, 0) -- (9.27, 0);

\fill (0, 6.0) -- (-0.1, 5.8) -- (0.1, 5.8);
\fill (9.27, 0) -- (9.07, 0.1) -- (9.07, -0.1);
\draw (9.27, -0.5) node {$x_1$};
\draw (0.65, 5.8) node {$x_2$};
\draw (-0.3, -0.35) node {$O$};

\draw[dashed, domain=30:150] plot ({2.1650635*cos(\x)},{-1.08253175+2.1650635*sin(\x)});
\draw[dashed, domain=-150:-30] plot ({2.1650635*cos(\x)},{1.08253175+2.1650635*sin(\x)});
\draw[dashed, domain=-40:220] plot ({0+2.40117*cos(\x)}, {1.5+2.40117*sin(\x)});
\draw[dashed, domain=140:400] plot ({0+2.40117*cos(\x)}, {-1.5+2.40117*sin(\x)});
\draw [dashed, domain=-5:70] plot ({5.34*cos(\x)}, {5.34*sin(\x)-5});
\draw [dashed, domain=175:250] plot ({5.34*cos(\x)}, {5.34*sin(\x)+5});
\draw [dashed, domain=110:185] plot ({5.34*cos(\x)}, {5.34*sin(\x)-5});
\draw [dashed, domain=290:365] plot ({5.34*cos(\x)}, {5.34*sin(\x)+5});

\fill (1.875, 0) circle (0.07);
\fill (-1.875, 0) circle (0.07);
\fill (0, 0) circle (0.07);

\draw (5, 0.46) node {$\xi=\xi_1$};
\draw (8.1, 1.45) node {$\xi=\xi_2$};

\draw (3.1098, 0.173648) -- (3.7098, 0.4);
\draw (6.29006, 1.25) -- (6.89006, 1.45);

\fill (3.1098, 0.173648) -- (3.1098+0.156693501419659, 0.173648+0.159521618011000) -- (3.1098+0.221633395260595, 0.173648-0.0296418303291268);
\fill (6.29006, 1.25) -- (6.29006+0.156693501419659, 1.25+0.159521618011000) -- (6.29006+0.221633395260595, 1.25-0.0296418303291268);

\end{tikzpicture}

\end{center}
\vskip -4mm
\caption{\label{fig:coord}$\xi$-level curves (thick) and $\theta$-level curves (dashed) of the bipolar coordinate system. We rotate and translate the original annulus (Figure \ref{fig:geometry}) such that $\p B_1^t$ and $\p B_2$ become $\xi$-level curves of some positive values ($0<\xi_2<\xi_1$).}
\end{figure}
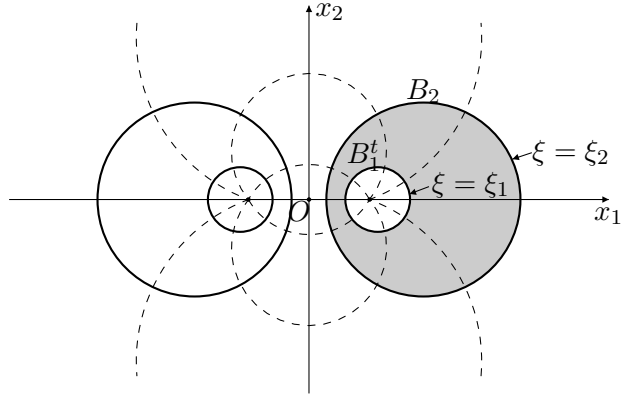

\subsection{Bipolar coordinates}
For $x=(x_1,x_2)$ in Cartesian coordinates, we define bipolar coordinates  $(\xi, \theta) \in \mathbb{R}\times(-\pi,\pi]$ via the relation
\beq\label{bipolar:z}
	(x_1,x_2)=\left( \frac{\alpha\sinh \xi}{\cosh \xi + \cos \theta}\,,\, \frac{\alpha\sin \theta}{\cosh \xi + \cos \theta}\right)
	\eeq
with the poles located at $(\pm\alpha,0)$, where $\alpha>0$ will be defined later depending on the parameter $t$.
We write $x=x(\xi,\theta)$ to indicate its dependence on $(\xi,\theta)$, if necessary. The scale factors for the parameters $\xi$ and $\theta$ coincide, given by
\beq\label{def:scale}
h_\xi=h_\theta=\frac{\alpha}{\cosh \xi+\cos\theta}.
\eeq
The coordinate level curves $x(\xi,\cdot)$ and $x(\cdot,\theta)$ define a curvilinear orthogonal frame in $\RR^2$.
The $\xi$-level curves of positive values are circles in the right half-plane, and the limiting cases $\xi=\pm\infty$ correspond to the poles $(x_1,x_2)=(\pm\alpha,0)$. 
The general form of the harmonic function in bipolar coordinates, by the method of separation of variables, is
\beq\label{SV:general} 
u(x)=a_0+b_0\xi+\sum_{n=1}^\infty\left((a_ne^{n\xi}+b_ne^{-n\xi})\cos(n\theta)+(c_ne^{n\xi}+d_ne^{-n\xi})\sin(n\theta)\right),
\eeq
where $a_n$, $b_n$, $c_n$ and $d_n$ are constant coefficients.
For a fixed $\tilde{\xi}>0$, the unit normal vector at $x(\tilde{\xi},\theta)$ to the circle $\xi=\tilde{\xi}$, outward with respect to the center of the circle, is
\[\nu_{\tilde{\xi}}=\left(-\frac{1+\cosh\tilde{\xi}\cos\theta}{\cosh\widetilde{\xi}+\cos\theta},\, \frac{\sinh\tilde{\xi}\sin\theta}{\cosh\tilde{\xi}+\cos\theta}\right)\]
and 
\beq\label{eqn:bipolar:normal}
\pd{u}{\nu_{\tilde{\xi}}}=-\frac{\cosh\tilde{\xi}+\cos\theta}{\alpha}\,\pd{u}{\xi}\Bigr|_{\xi=\tilde{\xi}}\,.
\eeq

A rigid motion on a domain does not change its Steklov--Dirichlet eigenvalues. 
We rotate and translate the annulus $\Om$ (see Figure \ref{fig:geometry}) and choose an appropriate $\alpha>0$ such that the inner and outer boundaries of the annulus become $\xi$-level curves of some positive values, namely, $\xi_1$ and $\xi_2$, respectively (see Figure \ref{fig:coord}).
Again, we call the inner disk, the outer disk and the annulus as $B_1^t$, $B_2$ and $\Om$, respectively.  
They now satisfy
\beq\label{eqn:B1B2:2}
B_1^t=t_0e_1+B(-te_1,r_1),\quad B_2=t_0e_1+B(0,r_2)\quad\mbox{for some }t_0>0.
\eeq
One can find that (see \cite{Ammari:2005:GES, Kim:2018:EFC} for the derivation)
\beq\label{def:alpha}
\alpha=\frac{\sqrt{(r_2+r_1)^2-t^2}\sqrt{(r_2-r_1)^2-t^2}}{2t}
\eeq
and
\beq\label{def:xi_j}
\xi_j=\ln\left(\frac{\alpha}{r_j}+\sqrt{\left(\frac{\alpha}{r_j}\right)^2+1}\right),\ j=1,2.
\eeq
We note that $0<\xi_2<\xi_1$ and that the interior of $\Om$ corresponds to the rectangular region $\xi_2<\xi<\xi_1$.
For later use, we denote $\ep$ the distance between the inner and outer boundaries of $\Om$. In other words, $$\ep:=r_2-r_1-t.$$
If the two boundaries of $\Om$ are close to each other (i.e., $\ep$ is small), we have (see \cite{Kim:2018:EFC})
\begin{align}
\alpha&=r_*\sqrt{\ep}+O(\ep\sqrt{\ep})\label{alpha:small:ep},\\
\label{xi:small:ep}\xi_j&=\frac{r_*}{r_j}\sqrt{\ep}+O(\ep\sqrt{\ep}),\ j=1,2, \ \mbox{with }r_*=\sqrt{\frac{2r_1r_2}{r_2-r_1}}.
\end{align}

\subsection{Series expansion of the first eigenfunction}\label{sec:series:eigenfun}
We can analytically extend $u_1^t$ across the boundary circles $C_1$ and $C_2$ on which the zero Dirichlet condition and the Robin boundary condition are assigned, respectively. 
We remind the reader that $u_1^t$ does not have a sign change in $\Om=B_2\setminus\overline{B_1^t}$ and $\sigma_1^t$ is simple (see Lemma \ref{positive}). 
The eigenfunction admits the expansion
\begin{align}\label{sov} 
u_1^t(x)&=a_0+b_0\xi+\sum_{n=1}^\infty\left(a_ne^{n\xi}+b_ne^{-n\xi}\right)\cos(n\theta)
\end{align}
for some constant coefficients $a_n$ and $b_n$.
Indeed, since $\Om$ is symmetric with respect to $x_1$-axis, $u_1^t(x_1,-x_2)$ is also an eigenfunction corresponding to $\sigma_1^t$. It therefore holds that
$$u_1^t(x_1,-x_2)=Cu_1^t(x_1,x_2)$$
for some constant $C$. Evaluating both sides on $x_2=0$ (where $u_1^t(x_1,x_2)$ and $u_1^t(x_1,-x_2)$ coincide and are non-zero from Lemma \ref{positive}), we have $C=1$. In other words, $u_1^t(x_1,x_2)$ is an even function with respect to the $x_2$-variable and, therefore, it is even with respect to $\theta$. Hence, in view of the general solution \eqnref{SV:general}, we obtain \eqnref{sov}.
In fact, the only unknowns are $a_n$ because of the following relation from the vanishing condition on $\p B_1^t$ (or, $\xi=\xi_1$):
 \beq\label{linsys}
\begin{cases}
\ds
&a_0+b_0\xi_1=0,\\[2mm]
\ds&a_n e^{n\xi_1}+b_ne^{-n\xi_1}=0\quad\text{for all } n\ge1.
\end{cases}
\eeq
\begin{notation}
For notational simplicity, we define
\begin{align}
A_n(t)&=na_ne^{n\xi_1},\label{def:An}\\
\widetilde{A}_n(t)&=na_ne^{n\xi_1}\cosh(n(\xi_1-\xi_2)),\label{def:Antilde}\\ 
F_n(t)
&=\frac{\widetilde{A}_{n+1}}{\widetilde{A}_n}\label{def:Fn},\\
 \label{def:Tn}
T_n(t) &=2\cosh\xi_2- \frac{2\alpha\sigma_1^t}{n} \tanh(n(\xi_1-\xi_2)) \quad\mbox{for each }n\geq1 .
\end{align}
\end{notation}

\begin{lemma} \label{coefficients}
Let $\sigma_1^t$ and $u_1^t$ be the Steklov--Dirichlet eigenfunction and the associated eigenfunction on $\Om=B_2\setminus\overline{B_1^t}$. Then, we have
\begin{align}
\label{seriesfirst}
&u_1^t(x) = a_0 -\frac{a_0}{\xi_1}\xi-\sum_{n=1}^{\infty} \frac{2}{n}A_n \sinh(n(\xi_1-\xi))\cos (n\theta),\\
&\left.\frac{\partial u_1^t}{\partial \xi}\right|_{\xi=\xi_2} = -\frac{a_0}{\xi_1}+\sum_{n=1}^{\infty} 2\widetilde{A}_n \cos({n\theta})\label{seriesnormalfirst}
\end{align}
and
\begin{align}
\label{expan:shapederi}
\frac{d}{dt}\sigma_1^t
&=\frac{2\pi}{\alpha}\bigg(-\frac{a_0^2}{\xi_1^2}+\frac{2a_0}{\xi_1}A_1\cosh\xi_1-2\sum_{n=1}^\infty \left(A_n^2+A_nA_{n+1}\cosh\xi_1\right)\bigg).
\end{align}
\end{lemma}
\begin{proof}
Using \eqnref{sov} and \eqnref{linsys}, 
$u_1^t(x)$ is expressed as \eqnref{seriesfirst}. It then follows that \eqnref{seriesnormalfirst}.
From \eqnref{int:shape:deri}, \eqnref{def:scale} and \eqnref{eqn:bipolar:normal}, we have
\begin{align*}
\frac{d}{dt}\sigma_1^t&=-\int_{\p B_1^t}\left(-\frac{\p u_1^t}{\p\nu_{\xi_1}}\right)^2(-\nu_{\xi_1}\cdot e_1) \, dS\\
&=-\frac{1}{\alpha}\int_{-\pi}^\pi\left(\frac{\p u_1^t}{\p\xi}\Bigr|_{\xi=\xi_1}\right)^2(1+\cosh\xi_1\cos\theta)\, d\theta.
\end{align*}
By applying \eqnref{seriesfirst}, we prove the lemma.
\end{proof}

\begin{lemma}\label{not:coefficient relations}
We can express $\widetilde{A}_n$ (or $a_n$) in terms of $r_1,r_2,t,\sigma_1^t$ by the recursive relation:
\beq\label{tA:recur}
\begin{cases}
\ds
&\widetilde{A}_1 = \ds a_0\frac{\cosh \xi_2}{\xi_1} - a_0\alpha\sigma_1^t\left( 1-\frac{\xi_2}{\xi_1} \right) , \\
\ds&\widetilde{A}_2 = \ds\frac{a_0}{\xi_1}
 +2\alpha\sigma_1^t \widetilde{A}_1\tanh(\xi_1-\xi_2)-2 \widetilde{A}_1\cosh \xi_2,\\
 \ds &\widetilde{A}_{n+2}= \ds -\widetilde{A}_{n+1}T_{n+1} - \widetilde{A}_{n},\quad n\geq 1. 
\end{cases}
\eeq
 The constant term $a_0$ is determined from the normalization condition \eqnref{eqn:normalized}. 
We remark that for other eigenvalues $\sigma$, the formulas \eqnref{sov} and \eqnref{tA:recur} also hold with $\sigma$ instead of $\sigma_1^t$.

 \end{lemma}

\begin{proof}

On $\p B_2$, it holds from the Robin boundary condition and \eqnref{eqn:bipolar:normal} that
\begin{align*}
\sigma_1^t u_1^t\big|_{\xi=\xi_2}=\frac{\partial u_1^t}{\partial \nu}\Big|_{\xi=\xi_2} =-\frac{\cosh \xi_2+\cos \theta}{\alpha} \, \frac{\partial u_1^t}{\partial \xi}\Big|_{\xi=\xi_2}. 
\end{align*} 
We have
\begin{align*}
\sigma_1^t u_1^t\big|_{\xi=\xi_2}= a_0 \sigma_1^t \Big(1-\frac{\xi_2}{\xi_1}\Big)
-2 \sigma_1^t \widetilde{A}_1\tanh(\xi_1-\xi_2) \cos\theta
-\sum_{n=2}^{\infty} \frac{2\sigma_1^t \widetilde{A}_n}{n} \tanh(n(\xi_1-\xi_2))\cos( n\theta).
\end{align*}
From \eqnref{seriesnormalfirst}, we also obtain
\begin{align} 
& -\left(\cosh \xi_2+\cos \theta\right) \, \frac{\partial u_1^t}{\partial \xi}\Big|_{\xi=\xi_2} \nonumber\\\notag
&=\frac{a_0\cosh \xi_2}{ \xi_1} +\frac{a_0 \cos \theta}{\xi_1} 
 -\sum_{n=1}^{\infty}\left(2\, \widetilde{A}_n \cosh \xi_2 \cos (n\theta )
+ \widetilde{A}_n\cos ((n-1)\theta) + \widetilde{A}_n\cos ((n+1)\theta )\right)\\
&=\frac{a_0 \cosh \xi_2}{ \xi_1} - {\widetilde{A}_1}
+ \big(\frac{a_0}{\xi_1} -{2\,\widetilde{A}_1\cosh \xi_2} - {\widetilde{A}_2}\big)\cos \theta-
 \sum_{n=2}^{\infty}\big({\widetilde{A}_{n-1}}+{2\, \widetilde{A}_n\cosh \xi_2} + {\widetilde{A}_{n+1}}\big)  \cos (n\theta). \label{L:comp}
\end{align}
We prove the lemma by comparing the two series.
\end{proof}

\subsection{Limit behavior of the ratio of consecutive coefficients $\widetilde{A}_n$ (i.e., $F_n$)}\label{sec:Proof:Thm1.3}

From \eqnref{tA:recur}, it holds that
 \beq\label{relation:F_n}
F_n=-T_n-\frac{1}{F_{n-1}}\quad\mbox{for all }n\geq2 \mbox{ such that }a_n\neq0.
\eeq 
Note that $T_n$ is not defined by $a_n$ but is explicitly defined in terms of elementary functions.
We first show some basic behaviors of $T_n$ and $a_n$ for sufficiently large $n$ as follows. The proofs of the lemmas in this subsection are provided in subsection \ref{sec:lemma:fixedpt}.
\begin{lemma}\label{lem:Tn:a_n}
\begin{itemize}
\item[\rm(a)] For any fixed $t\in[0, r_2-r_1)$, we can choose $n_0=n_0(t)\geq 2$ such that
\beq\label{Tn:bigger:2}
T_n(t) >2\quad\mbox{for all }n\geq n_0.
\eeq
\item[\rm(b)] For $n_0$ satisfying \eqnref{Tn:bigger:2}, it holds that
$a_n(t)\neq 0$ for all $n\geq n_0.$
\end{itemize}
\end{lemma}

We now analyze the convergence of $F_n$, with the aim of understanding the limit behavior of $\widetilde{A}_n$ (or $a_n$), as $n$ goes to infinity. To state the result, we define some terminologies:

\begin{itemize}
\item For fixed $t$ and $T_n>2$ (i.e., $n\geq n_0$), the system of two equations
$
 x_2=-T_n-\frac{1}{x_1},\ 
x_2=x_1
$
has the two intersections $(L_n, L_n)$ and $(U_n,U_n)$ with (see the left graph in Figure \ref{fig:graph})
\begin{align}\label{def:L_n}
L_n=\frac{-T_n-\sqrt{T_n^2-4}}{2},\quad
U_n=\frac{-T_n+\sqrt{T_n^2-4}}{2}.
\end{align}

\item 
For the limiting case, the graphs of 
$
x_2=-T_\infty-\frac{1}{x_1},\ 
x_2=x_1
$
has the two intersections $(L_\infty,L_\infty)$ and $(U_\infty,U_\infty)$ with
$$L_\infty = -e^{\xi_2},\quad U_\infty =-e^{-\xi_2}.$$
One can derive these values from \eqnref{def:L_n} with $T_n$ replaced by $T_\infty=2\cosh\xi_2$.
 It holds (see the right graph in Figure \ref{fig:graph}) that for all $n\ge n_0$,
\begin{align}\label{nest}
L_\infty<L_{n+1}<L_{n}<U_n< U_{n+1}<U_\infty<0.
\end{align}

\item For the case $T_n\leq2$ (i.e., $n\leq n_0-1$), we define
$L_n=U_n=-1.$

\end{itemize}

%
%
\begin{figure}[t!]
\vskip 2.5cm
\begin{center}
\begin{minipage}[t]{.45\linewidth}
\hspace{-.3cm}
\begin{tikzpicture}[transform canvas={scale=0.75}, every node/.style={scale=1.2}]
\draw[<->] (0,3) node[above] {$x_2$} -- (0,0)--(4.4,0) node[right]{$x_1$};
\draw (-4.4,0)--(0,0) --(0,-4.4);

 \draw[blue, thick] (3,3) --(-4.4,-4.4);
\draw plot[domain=-4.4:-0.2,smooth](\x,{-1/\x-2.03});
\draw plot[domain=0.43:4.4,smooth](\x,{-1/\x-2.03});
\draw[dotted] plot[domain=-4.4:4.4](\x, -2.03);
\node[below left] at (0,-2.03) {{$-T_n$}};
\node at (-3.6, -0.9) {$\displaystyle x_2=-T_n-\frac{1}{x_1}$};
\node at (-1.18885-0.15,0.2){{$L_n$}};
\node at (-0.841147+0.05,0.2) {{$U_n$}};
\draw[dotted] (-1.18885,-1.18885)--(-1.18885,0)  ;
\draw[dotted] (-0.841147,-0.841147)--(-0.841147,0) ;
\fill [black] (-1.18885,-1.18885) circle (2pt);
\fill [black] (-0.841147,-0.841147) circle (2pt);

\node[right,blue] at (2.7,2.5){$x_2=x_1$};
\end{tikzpicture}
\end{minipage}

\hspace{-16.4cm}
\begin{minipage}[t]{.45\linewidth}

\vspace{3.7cm}
\begin{tikzpicture}[transform canvas={scale=0.75},  every node/.style={scale=1.2}]
\begin{tikzpicture}[scale=5]
\draw (-0.5,-0.3)--(-1.6,-0.3);
\draw[blue, thick] (-0.5,-0.5) --(-1.6,-1.6);
\draw plot[domain=-1.6:-0.6,smooth](\x,{-1/\x-2.03});
\draw plot[domain=-1.6:-0.6,smooth](\x,{-1/\x-2.1});
\draw plot[domain=-1.6:-0.6,smooth](\x,{-1/\x-2.1999999976});

\node at (-1.6-0.385, 1/1.6-2.01) {$x_2=-T_n-1/x_1$};
\node at (-1.6-0.35, 1/1.6-2.1) {$x_2=-T_{n+1}-1/x_1$};
\node at (-1.6-0.376, 1/1.6-2.1999999976)  {$x_2=-T_{\infty}-1/x_1$};

\draw[dotted] (-1.18885,-1.18885)--(-1.18885,-0.3)  ;
\node at (-1.18885,-0.25){$L_n$};
\draw[dotted] (-1.37016,-1.37016)--(-1.37016,-0.3)  ;
\node at (-1.37016,-0.25) {$L_{n+1}$};
\draw[dotted] (-1.5582575654,-1.5582575654)--(-1.5582575654,-0.3);
\node at (-1.5582575654,-0.25) {$L_{\infty}$};
\draw[dotted] (-0.841147,-0.841147)--(-0.841147,-0.3) ;
\node at (-0.841147-0.05,-0.25) {$U_n$};
\draw[dotted] (-0.729844,-0.729844)--(-0.729844,-0.3);
\node at (-0.739844,-0.25) {$U_{n+1}$};
\draw[dotted] (-0.6417424322,-0.6417424322)--(-0.6417424322,-0.3);
\node at (-0.6317424322+0.05,-0.25) {$U_{\infty}$};

\node[right,blue] at (-0.56,-0.6) {$x_2=x_1$};
\fill [black] (-1.18885,-1.18885) circle (0.5pt);
\fill [black] (-0.841147,-0.841147) circle (0.5pt);
\fill [black] (-1.37016,-1.37016) circle (0.5pt);
\fill [black] (-0.729844,-0.729844) circle (0.5pt);
\fill [black] (-0.6417424322,-0.6417424322) circle (0.5pt);
\fill [black] (-1.5582575654,-1.5582575654) circle (0.5pt);

\end{tikzpicture}
\end{tikzpicture}
\end{minipage}
\end{center}
\vskip -8cm

\caption{Illustration of $L_n$ and $U_n$ (left). It holds that $L_{n+1}<L_n<U_n<U_{n+1}$ (right).}\label{fig:graph}
\end{figure}
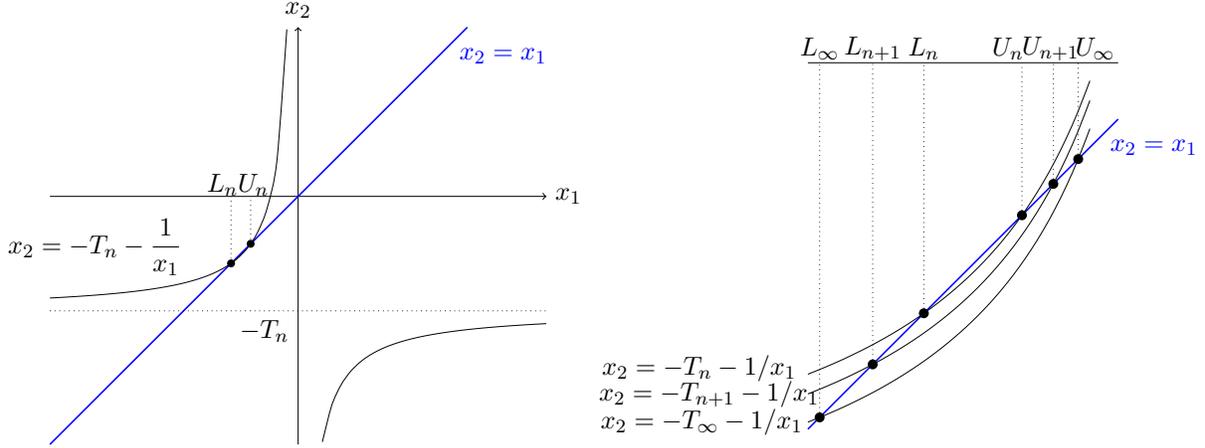

\begin{lemma}\label{lemma:fixedpt}
Fix $t$ in $[0, r_2-r_1)$ and let $n_0\in\NN$ satisfy \eqnref{Tn:bigger:2}.
We have two alternatives for the limit of $F_n$:
\beq\notag
\lim_{n\rightarrow\infty} F_n =
\begin{cases}
U_\infty \quad&\mbox{if }U_{n+1}<F_n<U_\infty\mbox{ for all }n\geq n_0,\\
L_\infty\quad&\mbox{otherwise}.
\end{cases}
\eeq
\end{lemma}

\noindent{\textbf{Proof of Theorem \ref{cor:coefficient relations}}}\hskip .3cm
The right-hand side of \eqnref{seriesfirst} converges at $\xi$, satisfying $\xi_2\leq\xi\leq \xi_1$. Hence, we have
\begin{align}\notag
1&\geq\limsup_{n\rightarrow\infty}\left|\frac{A_{n+1}\sinh((n+1)(\xi_1-\xi))}{A_n\sinh(n(\xi_1-\xi))}\right|\\ \notag
&=\limsup_{n\rightarrow\infty}\left|\frac{\widetilde{A}_{n+1}}{\widetilde{A}_n}\frac{\cosh(n(\xi_1-\xi_2))}{\cosh((n+1)(\xi_1-\xi_2))}\frac{\sinh((n+1)(\xi_1-\xi))}{\sinh(n(\xi_1-\xi))}\right|\\ \label{limsup:ratio:A_n}
&=\limsup_{n\rightarrow\infty}\left|F_n\right|e^{-\xi+\xi_2}.
\end{align}
Therefore, by taking $\xi=\xi_2$, one can eliminate $-e^{\xi_2}$ from the two alternatives of the limit of $F_n$ in Lemma \ref{lemma:fixedpt}, and the theorem follows.
\qed

From the alternatives in Lemma \ref{lemma:fixedpt} and Theorem \ref{cor:coefficient relations}, we directly obtain the following. 
\begin{cor}\label{cor:F_n}
For $n_0\in\NN$ satisfying \eqnref{Tn:bigger:2}, it holds that
\beq\label{eq:F_n:estimate}
U_{n+1}<F_n<U_\infty\quad\mbox{for all }n\geq n_0.
 \eeq
\end{cor}

\smallskip

Let us denote by $M(t)$ the right-hand side of \eqnref{eqn:upperbound:global:int}. Because of $\sigma_1^t\le\sigma_1^0$, we thus have $\sigma_1^t\leq \min\{M(t), \sigma_1^0\}$. 
From the definition of $T_n$, \eqnref{def:Tn}, and the fact that $\cosh\xi_2=\sqrt{(\frac{\alpha}{r_2})^2+1}$, it holds that for each $t\in(0,r_2-r_1)$,
\begin{align*}
T_n(t)&> 2\cosh\xi_2-\frac{2\alpha\sigma_1^t}{n}\\
&\ge 2\sqrt{\big(\frac{\alpha}{r_2}\big)^2+1}-\frac{2\alpha}{n} \min\left\{M(t),\sigma_1^0\right\}.
\end{align*}
Hence, equation \eqnref{Tn:bigger:2} (and thus \eqnref{eq:F_n:estimate}) holds for $n_0$ satisfying
\beq\label{n0:cond}
n_0\geq \frac{\alpha}{\sqrt{\big(\frac{\alpha}{r_2}\big)^2+1}-1}\, \min\left\{M(t), \sigma_1^0\right\}.
\eeq

\subsection{Proofs of Lemmas \ref{lem:Tn:a_n} and \ref{lemma:fixedpt}}\label{sec:lemma:fixedpt}
\noindent{\textbf{Proof of Lemma \ref{lem:Tn:a_n}}}\hskip .3cm
For all $t$ fixed in $[0,r_2-r_1)$, $T_n(t)$ is monotone increasing to $T_\infty(t)= 2\cosh \xi_2>2$ as $n\to\infty$. Hence, we can choose $n_0$ such that \eqnref{Tn:bigger:2} holds. Hence, we have (a).

Note that $a_0+b_0\xi$ can satisfy \eqnref{problem} with $\sigma>0$ only when $a_0=b_0=0$. Hence, $a_n$ is nonzero for some $n\geq1$. It implies that any two consecutive coefficients $a_n$ and $a_{n+1}$ cannot both be zero due to the recursive relation \eqnref{tA:recur}. 

Suppose that $a_{n}=0$ (i.e., $\widetilde{A}_n=0$) for some $n\geq n_0$. Then, we have $a_{n+1}\neq0$ (i.e., $\widetilde{A}_{n+1}\neq0$). From \eqnref{tA:recur} and (a), it holds that 
$|\widetilde{A}_{n+2}|>|\widetilde{A}_{n+1}|.$
We also have
$|\widetilde{A}_{n+3}|>2|\widetilde{A}_{n+2}|-|\widetilde{A}_{n+1}|>|\widetilde{A}_{n+2}|.$
Inductively, it holds that
$|\widetilde{A}_{n+j+1}|>|\widetilde{A}_{n+j}|$ for all $j\geq 0$. 
It contradicts the fact that the series \eqnref{seriesnormalfirst} is convergent. 
Hence, we have (b).
\qed

\smallskip

\noindent{\textbf{Proof of Lemma \ref{lemma:fixedpt}}}\hskip .3cm
 We prove the lemma by case-by-case discussions. 
We consider, first, the cases satisfying $F_{n_0}\leq U_{n_0+1}$ (Case 1) and, second, the cases satisfying $U_{n_0+1}<F_{n_0}$ (Case 2).  For only Case 2-2, $F_n$ converges to $U_\infty$.

The following is useful in the derivation: for each $n\geq n_0$,
\beq\label{x1:smaller:L_n}
\begin{cases}
\ds x_1 = -T_n -\frac{1}{x_1}\quad& \mbox{for }x_1=L_n, U_n,\\[2mm]
\ds x_1< -T_n-\frac{1}{x_1}\quad&\mbox{for }x_1\in(-\infty, L_n)\cup(U_n,0),\\[2mm]
\ds L_n<-T_n-\frac{1}{x_1}<x_1<U_n\quad&\mbox{for } L_n<x_1<U_n.
\end{cases}
\eeq

\noindent{\bf{Case 1-1}} ($F_{n}<L_{n+1}$ for all $n\geq n_0$).
From \eqnref{x1:smaller:L_n} (the second case with $x_1=F_n$), $F_n$ is monotone increasing in $n$:
\begin{align}\label{increasingsequence}
F_{n+1} = -T_{n+1}-\frac{1}{F_{n}}>F_{n}\quad\mbox{for all }n\geq n_0.
\end{align}
On the other hand, as $L_n$ is monotone decreasing to $L_\infty$ and $F_n<L_{n+1}$, $F_n$ is bounded above. Hence, $F_n$ is convergent and its limit is bounded above by $L_\infty$. We set $F_\infty:=\lim_{n\rightarrow\infty}F_n\leq L_\infty$.
In fact, we have
$$F_\infty= L_\infty.$$
If not, then $F_\infty<L_\infty$ and, hence,
$F_\infty<-T_\infty-\frac{1}{F_\infty}$. Set $\delta :=  -T_\infty-\frac{1}{F_\infty}-F_\infty>0$. As $F_n$ is negative and converges to a negative number, so is its reciprocal and, hence, for sufficiently large $n$, we have $\left|\frac{1}{F_n}-\frac{1}{F_\infty}\right|<\frac{\delta}{2}$.
Since $T_n$ is monotone increasing to $T_\infty$, we obtain
\begin{align}\label{lemma:fixedpt:ineq1}
F_{n+1} = -T_{n+1} -\frac{1}{F_n}> -T_\infty -\frac{1}{F_n}>-T_\infty-\frac{1}{F_\infty}-\frac{\delta}{2}=F_\infty+\frac{\delta}{2}.
\end{align}
This contradicts that $F_n$ converges to $F_\infty$. Hence, $\lim_{n\rightarrow\infty}F_n=L_\infty$.

\smallskip
\smallskip

\noindent{\bf{Case 1-2}} ($F_{n_0}<L_{n_0+1}$, but $F_{n} \ge L_{n+1}$ for some $n>n_0$).
Let $n_1$ be the smallest integer satisfying $n_1>n_0$ and $F_{n_1} \ge L_{n_1+1}$. 
 Then,  from \eqnref{nest}, we have $F_{n_1-1} <L_{n_1}<0$ and
\begin{align*}
F_{n_1} = -T_{n_1} -\frac{1}{F_{n_1-1}} < -T_{n_1} -\frac{1}{L_{n_1}} =  L_{n_1}<U_{n_1}<U_{n_1+1}.
\end{align*}
Hence, $L_{n_1+1}\leq F_{n_1}<U_{n_1+1}$. 
This case reduces to either Case 1-3 or Case 1-4.

\smallskip
\smallskip

\noindent{\bf{Case 1-3}} ($L_{n_0+1}<F_{n_0}<U_{n_0+1}$). One can easily find from \eqnref{nest} and \eqnref{x1:smaller:L_n} that 
$$L_{n_0+2}<L_{n_0+1}<F_{n_0+1}=-T_{n_0+1}-\frac{1}{F_{n_0}}<F_{n_0}<U_{n_0+1}<U_{n_0+2}.$$
Inductively, it holds that
$L_{n+1}<F_{n}<U_{n+1}$ and $F_{n+1}<F_n$ for all $n\geq n_0$. 
Hence, $F_n$ is convergent and its limit is bounded below by $L_\infty$ and strictly smaller than $U_\infty$. We prove that the limit equals to $L_\infty$ by a similar procedure to that used in Case 1-1.

 If $F_\infty:=\lim_n F_n>L_\infty$, then $F_\infty>-T_\infty-\frac{1}{F_\infty}$. Set $\delta:= F_\infty+T_\infty+\frac{1}{F_\infty}>0$ and take a sufficiently large $\widetilde{n}$ such that $\left|\frac{1}{F_{\widetilde{n}}}-\frac{1}{F_\infty}\right|<\frac{\delta}{4}$ and $|T_{\widetilde{n}+1}-T_\infty|<\frac{\delta}{4}$.
 Then, we have
\begin{align} \label{lemma:fixedpt:ineq2}
F_{\widetilde{n}+1}=-T_{\widetilde{n}+1}-\frac{1}{F_{\widetilde{n}}}<-T_\infty-\frac{1}{F_\infty}+\frac{\delta}{2}=F_\infty-\delta+\frac{\delta}{2}<F_\infty.
\end{align} 
As $F_n$ is monotone decreasing and $F_{\widetilde{n}+1}<F_\infty$, $F_n$ cannot converge to $F_\infty$, a contradiction. Hence, $\lim_{n\rightarrow\infty}F_n=L_\infty$.

\smallskip
\smallskip

\noindent{\bf{Case 1-4}} ($F_{n_0} = L_{n_0+1}$ or $F_{n_0}=U_{n_0+1}$).
We have from \eqnref{nest} and the definition of $L_{n_0+1}$ or $U_{n_0+1}$ that 
$$L_{n_0+2}<F_{n_0+1} =F_{n_0} <U_{n_0+2}.$$ This case reduces to Case 1-3 and $\lim_{n\rightarrow\infty}F_n=L_\infty$.

\smallskip
\smallskip

\noindent{\bf{Case 2-1}} ($U_{n_0+1}<F_{n_0}$, but $F_{n_1}\le U_{n_1+1}$ for some $n_1>n_0$)).
This case reduces to Case 1 by taking $n_1$ instead of $n_0$.

\smallskip
\smallskip

\noindent{\bf{Case 2-2}} ($U_{n+1}<F_n< U_\infty$ for all $n\ge n_0$).
Since $U_{n+1}$ converges to $U_\infty$, $F_n$ also converges to $U_\infty$.

\smallskip
\smallskip

\noindent{\bf{Case 2-3}} ($U_\infty\leq F_{n_1}$ for some $n_1\geq n_0$). From Lemma \ref{lem:Tn:a_n}\,(b), $F_n\neq 0$ for all $n\geq n_0$. 
It further satisfies
\beq\label{F_n:negative}
0<F_n\quad\mbox{for some }n\geq n_0.
\eeq
If not, $F_n$ is negative for all $n\geq n_0$. Because of $U_{n_1+1}<U_\infty\leq F_{n_1}<0$, we have by induction that
$$F_{n+1}=-T_{n+1}-\frac{1}{F_n}>F_n\geq U_\infty\quad\mbox{for all }n\geq n_1.$$
Hence, $F_n$ is negative and increasing so that it converges to, namely, $F_\infty$ satisfying $U_\infty<F_\infty\leq 0$. 
As $F_{n+1}$ is a bounded sequence, $F_n$ cannot be very close to zero in view of \eqnref{relation:F_n}.
Hence, $U_\infty<F_\infty<0$ and, thus, $-T_\infty-\frac{1}{F_\infty}>F_\infty$.  
Set $\delta := -T_\infty-\frac{1}{F_\infty}-F_\infty>0$ and take a sufficiently large $\widetilde{n}$ such that $\left|T_\infty -T_{\widetilde{n}+1}\right|<\frac{\delta}{4}$ and $\left|\frac{1}{F_{\widetilde{n}}} -\frac{1}{F_\infty}\right|<\frac{\delta}{4}$. Then, it follows that
\begin{align}\label{lemma:fixedpt:ineq3}
F_{\widetilde{n}+1} = -T_{\widetilde{n}+1}-\frac{1}{F_{\widetilde{n}}}>-T_\infty-\frac{1}{F_\infty}-\frac{\delta}{2}>F_\infty. 
\end{align}
As $F_n$ is monotone increasing and $F_{\widetilde{n}+1}>F_\infty$, $F_n$ cannot converge to $F_\infty$, a contradiction. Therefore, we have \eqnref{F_n:negative}, and it implies from \eqnref{relation:F_n} that
$F_{n+1}<-T_{n+1}<L_{n+1}$ for some $n\geq n_0$. Hence, this case reduces to  Case 1-1 or Case 1-2.
\qed

\section{Asymptotic analysis assuming the small distance between the two boundaries of the annulus}\label{sec:asymp}
In this section, we assume that the distance between the two boundaries of the annulus $\Om$ is small, i.e.,
$$\ep=r_2-r_1-t\ll1.$$
Recall that $\alpha=r_*\sqrt{\ep}+O(\ep\sqrt{\ep})$. In view of \eqnref{n0:cond}, equation \eqnref{Tn:bigger:2} (and thus \eqnref{eq:F_n:estimate}) holds for $n_0$ satisfying
\beq\notag
n_0\geq \frac{3r_2^2\min\left\{M(t),\sigma_1^0\right\}}{r_*\sqrt{\ep}}.
\eeq

 \subsection{Asymptotic behavior of $T_n$ and $U_n$}
 Note that $\sigma_1^t$ is uniformly bounded independently of $t$ as it is positive for all $t$ and attains a maximum at $t=0$. 
Since $\alpha$ and $\xi_j$ are $O(\sqrt{\ep})$ and admit the asymptotics in \eqnref{alpha:small:ep} and \eqnref{xi:small:ep}, we have from Lemma \ref{not:coefficient relations} that
\beq\label{tildeA1}
\widetilde{A}_1=a_0\left(\frac{1}{\xi_1}+O(\sqrt{\ep})\right), \quad\widetilde{A}_2=a_0\left(-\frac{1}{\xi_1}+O(\sqrt{\ep})\right),
\eeq
and, thus,
\beq\label{F2:ep}
F_1(t)=\frac{\widetilde{A}_2}{\widetilde{A}_1}=-1+O(\ep).
\eeq
We also obtain
\begin{align}
T_2(t)& =2\cosh\xi_2- \frac{2\alpha\sigma_1^t}{2} \tanh(2(\xi_1-\xi_2)) \notag\\
&=2\cosh\xi_2 -2\alpha\sigma_1^t(\xi_1-\xi_2)\frac{\tanh(2(\xi_1-\xi_2))}{2(\xi_1-\xi_2)} \notag \\
&=2+\xi_2^2+O(\ep^2)-2\alpha\sigma_1^t(\xi_1-\xi_2)+O(\ep^2) \notag \\ \label{esti:T2}
&=2+\frac{2r_1}{r_2(r_2-r_1)}\ep-{4}\sigma_1^t \ep +O(\ep^2)
\end{align}
by using the property that $\frac{\tanh s}{s}=1+O(s^2)$ and $\frac{\tanh s}{s}\leq 1$ for all $s\geq 0$.
For general $n$ we have the following lemma, with which we derive a lower bound of $\sigma_1^t$ in subsection \ref{sec:asymp:lowerbound}. 

\begin{lemma}\label{Un:Rn:ep}
There exists $\ep_0>0$ independent of $n$ such that for all $\ep<\ep_0$ and $n\in\NN$,
\begin{align}
T_n(t) &=  2+\frac{2r_1}{r_2(r_2-r_1)}R_n(t)\ep+O(\ep^2),\label{Tn:asymp2}\\
U_n(t)&=-e^{-\sqrt{R_n^+(t)}\,\xi_2}+O(\ep), \label{Un:asymp:lemma}
\end{align}
where \beq\label{def:Rn}
R_n(t)=1-\sigma_1^{t}\,\frac{2r_2(r_2-r_1)}{r_1}\,\frac{\tanh (n(\xi_1-\xi_2))}{n(\xi_1-\xi_2)},
\eeq
and $O(\epsilon)$, $O(\ep^2)$ are uniform in $n$ and $a^+:=\max(a,0)$ for $a\in\RR$. 
\end{lemma}
\begin{proof}
We have the uniform boundedness for $R_n(t)$ thanks to $|\frac{\tanh s}{s}|\le 1$ and $\sigma_1^t \le \sigma_1^0$. 
In other words, there exists a positive constant $C$ independent of $\ep$ and $n$ such that 
\beq\label{rn:uniform}
\left|R_n(t)\right|\leq C\quad\mbox{for all }\ep>0, n\in\NN.
\eeq
We then estimate
\begin{align}
T_n(t) 
&=2\cosh \xi_2 -2\alpha(\xi_1-\xi_2)\sigma_1^t \frac{\tanh (n(\xi_1-\xi_2))}{n(\xi_1-\xi_2)} \notag\\
&= 2+\xi_2^2+O(\ep^2) -2\alpha(\xi_1-\xi_2)\sigma_1^t \frac{\tanh (n(\xi_1-\xi_2))}{n(\xi_1-\xi_2)} \label{Tn:asymp0}.
\end{align}
Here, $O(\ep^2)$ follows from $2\cosh\xi_2-2-\xi_2^2$, so it is uniform in $n$. From \eqnref{alpha:small:ep} and \eqnref{xi:small:ep}, it holds that \eqnref{Tn:asymp2}. In the remaining, we prove \eqnref{Un:asymp:lemma}. 

If $T_n(t)\leq 2$, then $U_n(t)=1$ by the definition of $U_n$ and $R_n^+(t)=O(\ep)$ from \eqnref{Tn:asymp2}. Hence, we have \eqnref{Un:asymp:lemma} for this case. 

If $T_n(t)>2$ and $R_n(t)<0$, then $2< T_n(t)\leq 2+O(\ep^2)$ from \eqnref{Tn:asymp2} and $U_n(t)=-1+O(\ep)$ from the definition of $U_n$ (see \eqnref{def:L_n}). Since $R_n^+(t)=0$, \eqnref{Un:asymp:lemma} holds.

If $T_n(t)>2$ and $R_n(t)\ge 0$, then \eqnref{Tn:asymp2} gives
\begin{align*}
U_n(t) &= \frac{-T_n(t)+\sqrt{T_n(t)^2-4}}{2} \\
&= -1+\sqrt{\frac{2r_1}{r_2(r_2-r_1)}}\sqrt{R_n(t)}\sqrt{\ep}+O(\ep) \\
&=-e^{-\sqrt{R_n(t)}\xi_2}+O(\ep),
\end{align*}
which is the desired conclusion. 
\end{proof}
 \subsection{Proof of Theorem \ref{prop:asymp:lowerbound} (lower bound of $\sigma_1^t$)}\label{sec:asymp:lowerbound}

Suppose, to the contrary, that there exists a constant $C$ satisfying $0<C<1$ and a sequence $\{\ep_j\}_{j=1}^\infty$ such that
$t_j=r_2-r_1-\ep_j$ is in $(0,r_2-r_1)$, $t_j\uparrow (r_2-r_1)$, and
\beq\label{lower:assump}
\sigma_1^{t_j}<C\frac{\ds r_1}{\ds2r_2(r_2-r_1)} \quad\mbox{for all }j.
\eeq
It then holds that
\begin{align*}
R_3(t_j)&=1-\sigma_1^{t_j}\,\frac{2r_2(r_2-r_1)}{r_1}\,\frac{\tanh (3(\xi_1^j-\xi_2^j))}{3(\xi_1^j-\xi_2^j)}
\geq 1-C >0\quad\mbox{for all }j,
\end{align*}
where $\xi_1^j$ and $\xi_2^j$ denote the level values $\xi_1$ and $\xi_2$ depending on $t_j$. 
Since $R_3(t_j)>0$, we have from \eqnref{Un:asymp:lemma} and \eqnref{rn:uniform} that 
\begin{align}\notag
U_3(t_j)
=-1+\sqrt{R_3(t_j)}\sqrt{\frac{2 r_1}{r_2(r_2-r_1)}}\sqrt{\ep_j}+O(\ep_j)\\
\label{U3:ep}
\geq -1+ \sqrt{1-C}\sqrt{\frac{2 r_1}{r_2(r_2-r_1)}}\sqrt{\ep_j}+O(\ep_j).
\end{align}
On the other hand, from \eqnref{esti:T2} and \eqnref{lower:assump}, we have $T_2(t_j)>2$ for sufficiently large $j$. 
As $T_n$ is increasing in $n$, we have $$T_n(t_j)>2\quad\mbox{for all }n\geq2.$$
It then follows from Corollary \ref{cor:F_n}, \eqnref{F2:ep} and \eqnref{esti:T2} that
\beq\label{U3:F2}
U_3(t_j)<F_2(t_j)= -T_2(t_j) -\frac{1}{F_1(t_j)}\leq-1+O(\ep_j).
\eeq
This contradicts \eqnref{U3:ep}. Hence, we prove the theorem. 
\qed

\begin{remark}\label{remark:DS}
For the domain $B(0,1)\setminus \overline{A}$ which is conformally equivalent to the Gr\"{o}tzsch ring $R_G(r):=B(0,1)\setminus ([0,r]\times\{0\})$ for some $0<r<1$,  Dittmar and Solynin showed in \cite{Dittmar:2003:MSE} that
\begin{align}\label{result:DS}
\sigma_1\left(B(0,1)\setminus \overline{A}\right) \ge \sigma_1\left(R_G(r)\right).
\end{align}
It is well known that $\sigma_1\left(R_G(r)\right)<\frac{1}{2}$ for all $r$. Therefore, we have
$$
\frac{r_1}{2r_2(r_2-r_1)} >\frac{1}{2}>\sigma_1\left(R_G(r)\right)\quad\mbox{for }r_2=1,r_1>0.5.$$
We highlight that the bound in \eqnref{lower:liminf} is larger than that in \eqnref{result:DS} when $\frac{r_1}{r_2}>\frac{1}{2}$ and the two boundaries of an annulus are sufficiently close. 
\end{remark}

\section{Numerical computation}\label{sec:numerical} 
We compute $\sigma_1^t$ in two dimensions by using the bipolar coordinates.
The numerical results support the conjecture that $\sigma_1^t$ is monotone decreasing as a function of $t$.
\subsection{Computation scheme}
 The Steklov--Dirichlet eigenvalue problem  \eqnref{eqn:laplacian}--\eqnref{BC2} is equivalent to the eigenvalue problem of the Dirichlet-to-Neumann operator
\begin{align}\label{def:oper:L}
    L :    \hat{u} &\mapsto \frac{\partial {u}}{\partial \nu}\Big|_{C_2}\quad\mbox{on }C^\infty(C_2),
\end{align}
where ${u}$ is the harmonic extension of $\hat{u}$, i.e., $\Delta u=0$ in $\Omega$, $u=0$ on $C_1$, and $u=\hat{u}$ on $C_2$.
The operator $L$ is positive-definite, self-adjoint with respect to the $L^2$ inner-product and has a discrete spectrum (see, for example, \cite{Agranovich:2006:MPS}).

We denote by $C_e^\infty(C_2)$ the collection of even functions in $C^\infty(C_2)$ with respect to the bipolar coordinate $\theta$. Each function in $C_e^\infty(C_2)$ admits a cosine series expansion in $\theta$. 
It directly follows from \eqnref{linsys} and \eqnref{L:comp} that 
\begin{align*}
& L\left[a_0 -\frac{a_0}{\xi_1}\xi_2-\sum_{n=1}^{\infty} 2a_n e^{\xi_1} \sinh(n(\xi_1-\xi_2))\cos (n\theta)\right]=\frac{1}{\alpha}  \bigg(\frac{a_0 \cosh \xi_2}{ \xi_1} - {\widetilde{A}_1}\\
&\qquad+ \Big(\frac{a_0}{\xi_1} -{2\,\widetilde{A}_1\cosh \xi_2} - {\widetilde{A}_2}\Big)\cos \theta-
 \sum_{n=2}^{\infty}\left({\widetilde{A}_{n-1}}+{2\, \widetilde{A}_n\cosh \xi_2} + {\widetilde{A}_{n+1}}\right)  \cos (n\theta) \bigg) \nonumber
\end{align*}
with $\widetilde{A}_n=na_{n} e^{n\xi_1}\cosh(n(\xi_1-\xi_2))$. 
One can replace the coefficients $a_n$ by any numbers as long as the series converges. 
In particular, we have
\begin{align*}
L\left[1\right]&=\frac{1}{\alpha(\xi_1-\xi_2)}\Big(\cosh\xi_2+\cos\theta\Big),\\
L\left[\cos\theta\right]&=\frac{1}{2\alpha\tanh(\xi_1-\xi_2)}\Big(1+2\cosh \xi_2\cos\theta+\cos(2\theta)\Big),\\
L\left[\cos (k\theta)\right]&=\frac{k}{2\alpha\tanh\left(k(\xi_1-\xi_2)\right)}\Big(\cos((k-1)\theta)+2\cosh\xi_2\cos (k\theta)+\cos((k+1)\theta)\Big),\ k\geq2.
\end{align*}
This implies that $L(C_e^\infty(C_2))\subset C_e^\infty(C_2)$. The first eigenfunction of $L$, $u_1^t$, is an even function as shown in Lemma \ref{coefficients}. Hence, $\sigma_1^t$ is the first eigenvalue of the restriction $L|_{C_e^\infty(C_2)}$ as well.

The main idea of the numerical computation of $\sigma_1^t$ is to consider the finite section operator $P_n L P_n$ with the projection operator $P_n$ from $C_e^\infty(C_2)$ onto the $n$-dimensional subspace
$H_n:=\mbox{span}\left\{{\cos (k\theta)}\,:\,k=0,1,\cdots,n-1\right\}.$ The finite section operators $P_nLP_n$, for example,
\begin{align*}
&P_3LP_3\left[\sum_{k=0}^\infty c_k\, {\cos(k\theta)}\right]
=
\left(
\begin{array}{c}
\ds 1\\
\cos(\theta)\\
\cos(2\theta)
\end{array}
\right)^T
\frac{1}{\alpha}
\left(
\begin{array}{ccc}
\ds \cosh\xi_2\cdot d_0^2  &\ds  d_1^2 & \ds0  \\[.5mm]
\ds d_0^2 & \ds 2\cosh \xi_2\cdot d_1^2  & \ds d_2^2  \\[.5mm]
 \ds0 & \ds d_1^2  &   \ds 2\cosh\xi_2\cdot d_2^2
\end{array}
\right)
\left(
\begin{array}{c}
c_0\\
c_1\\
c_2
\end{array}
\right),
\end{align*}
are tridiagonal with respect to the cosine functions. 
Furthermore, we define a new inner product on $C_e^\infty(C_2)$ by
\beq\label{H:inner}
\left(\cos(m\theta),\cos(k\theta)\right)=d_k^2\delta_{mk}
\eeq
 with
 \beq\label{def:v_n}
d_k^2=
\begin{cases}
\ds
\frac{1}{\xi_1-\xi_2} \quad&\mbox{if }k=0,\\[1.5mm]
\ds \frac{k}{2\tanh (k(\xi_1-\xi_2))} \quad&\mbox{otherwise}.
\end{cases}
\eeq
Here, $\delta_{mk}$ denotes the Kronecker's delta. 
In terms of the orthogonal basis $\left\{\frac{\cos(k\theta)}{d_k}\right\}$ of $C_e^\infty(C_2)$ equipped with the inner product \eqnref{H:inner}, the operator $P_n L P_n$ is identical to the symmetric matrix
\beq\label{M_n:full}
M_n = \frac{1}{\alpha}\left(
\begin{array}{ccccc}
\cosh\xi_2\cdot d_0^2 & d_0d_1 & &&\\[.5mm]
d_0 d_1&2\cosh\xi_2 \cdot d_1^2& d_1d_2&&\\[.5mm]
&d_1 d_2&2\cosh\xi_2\cdot d_2^2&\ddots&\\[.5mm]
&&\ddots&\ddots& d_{n-2}d_{n-1}\\[.5mm]
&&&d_{n-2}d_{n-1}&2\cosh\xi_2\cdot d_{n-1}^2
\end{array}
\right).
\eeq
We denote by $\sigma_{1,n}^t$ the first eigenvalue of $P_n L P_n$. We set $u_{1,n}^t$ to be a function in bipolar coordinates of series form \eqnref{seriesfirst} whose coefficients are given by the first eigenvector of $P_n L P_n$. As $P_nLP_n$ is identical to a finite dimensional matrix, one can easily compute $\sigma_{1,n}^t$ and $u_{1,n}^t$. 
\begin{lemma}\label{lemma:num:eig:decrease} 
For fixed $t$, $\{\sigma_{1,n}^t\}_{n=1}^\infty$ is a decreasing sequence of positive numbers.
\end{lemma}
From this lemma, $\sigma_{1,n}^t$ converges.
We then can derive lower and upper bounds of $\sigma_1^t$ by applying the variational formulation of the first eigenvalue for $L$ and $P_n L P_n$. 
We prove Lemma \ref{lemma:num:eig:decrease} and Proposition \ref{p:num:2d} in subsection \ref{sec:Proof:numer}.
\begin{prop}\label{p:num:2d} 
For any $m\in\mathbb{N}$, it holds that
\beq\label{2D:numeric:vs:exact}\lim_{n\to\infty}\sigma_{1,n}^t\le \sigma_1^t\le \frac{\ds\int_{\Om}\left|\nabla u_{1,m}^t\right|^2\, dx}{\ds \int_{\partial \Om}\left|u_{1,m}^t\right|^2\,dS}.\eeq
\end{prop}

In the following examples, we show the first Steklov--Dirichlet eigenvalue on various annuli, where the eigenvalue $\sigma_{1}^t$ is numerically computed by the following two-step procedure:
\begin{itemize}
\item {\textbf{Step 1.}} We numerically compute $\lim_{n\to\infty}\sigma_{1,n}^t$ by evaluating $\sigma_{1,n}^t$ with a sufficiently large truncation size $n$. More precisely, we iteratively compute $\sigma_{1,n}^t$, where $n$ is doubled from the initial value $2^3$ (i.e., $n=2^k$ for some $k\geq3$) until the stopping criterion  
\beq\label{relative:stop}
\left| \sigma_{1, 2^{k-1}}^t-\sigma_{1,2^{k}}^t\right|<10^{-12}
\eeq
is met. For all the examples in this section, this stopping condition is satisfied at $k\leq 8$. Table \ref{fig:2d:rel:errors} shows the relative error for some example annuli.

\item {\textbf{Step 2.}} Let $N=2^k$ be the truncation size obtained in Step 1, which satisfies \eqnref{relative:stop}. In this second step, we validate that $\sigma_{1,N}^t$ also approximates $\sigma_1^t$ as well as $\lim_{n\to\infty}\sigma_{1,n}^t$. We gradually increase the truncation size $m$ and evaluate the upper bound of $\sigma_{1}^t$ in \eqnref{2D:numeric:vs:exact}, namely $I^t_m$. For all the examples in this section, the difference between the upper bound $I^t_m$ and $\sigma_{1,N}^t$ decreases and eventually satisfies 
\beq\label{E_nN:condition}
E_{m,N}:=\left|\sigma_{1,N}^t-I^t_m  \right|<10^{-12}.
\eeq
Consequently, in view of \eqnref{2D:numeric:vs:exact}, $\sigma_{1,N}^t$ approximates $\sigma_1^t$.
Figure \ref{fig:plot:error} shows the graph of $E_{m,N}$ against $m$ for an example. 
\end{itemize}

\begin{figure}[H]
\centering
\hskip -0.5cm

\includegraphics[width=0.6\textwidth, trim=1cm 3.7cm 1cm 3.7cm, clip]{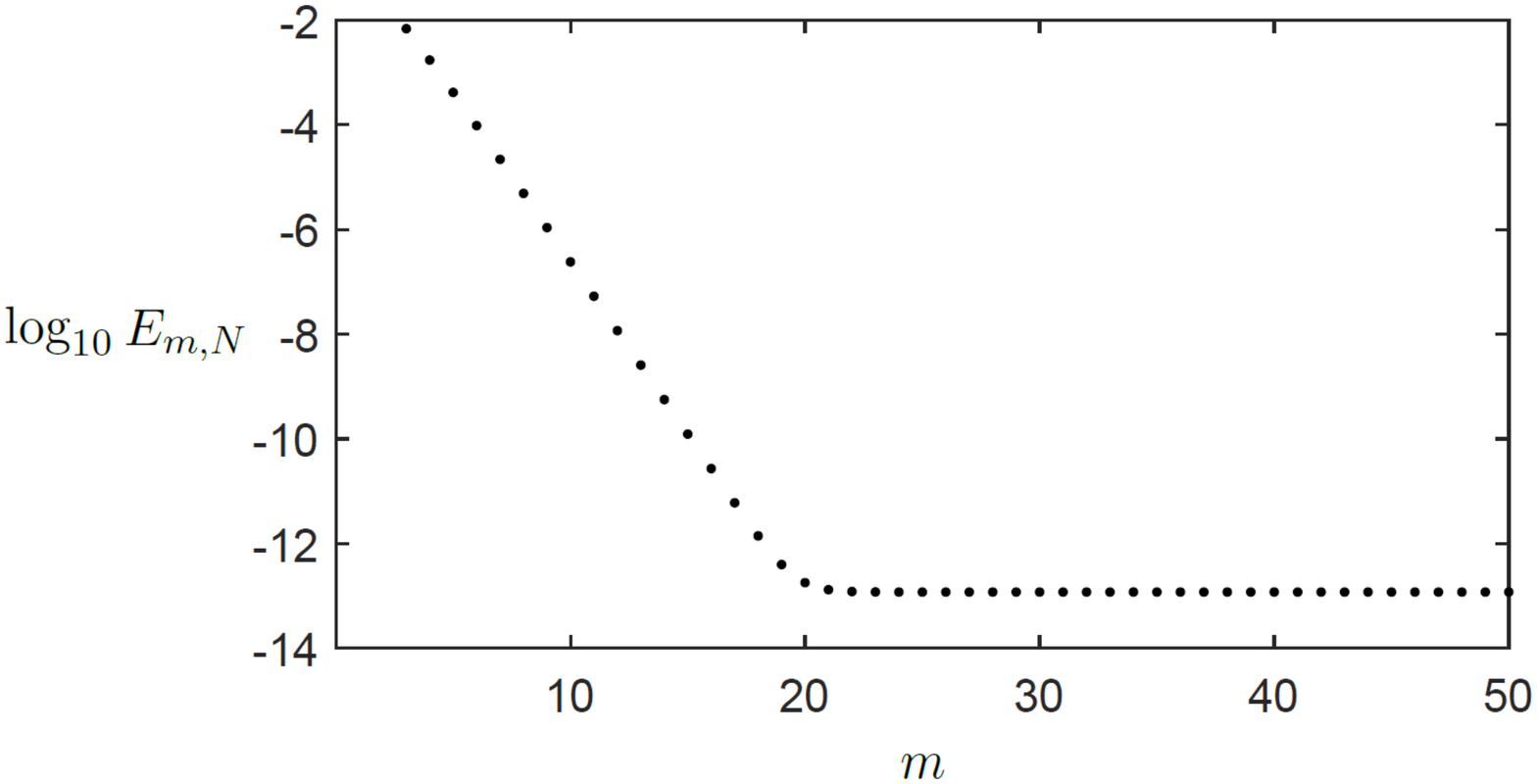}
\hskip 0.5cm
\vskip -.1cm
\caption{\label{fig:plot:error} 
The log-scale graph of $E_{m,N}$ against $m$ for $r_1=1$, $r_2=3$, $t=1.2$ and $N=2^6$,
where $E_{m,N}$ decreases exponentially until it reaches the relative error threshold \eqnref{relative:stop}.}
\end{figure}

\subsection{Examples}

To investigate the monotonicity of the first Steklov--Dirichlet eigenvalue on eccentric annuli, it is sufficient to consider the case $r_2=1$  because of
$\sigma_1(s\Om)=\frac{1}{s}\sigma_1(\Om)$ for $s>0.$
In the following examples, we visualize $\sigma_1^t$ for various annuli, where $r_2$ is fixed to be $1$. 
\begin{example}\label{example:1}
Figure \ref{fig:2d:numerics} plots $\sigma_1^t$ of the annulus in two dimensions with $r_1=1$, $r_2=3$ and 
$
\frac{t}{r_2-r_1}=0,\, 0.02,\, 0.04,\, \dots,\, 0.98\ (50\mbox{ cases}),
$
where the asymptotic lower bound obtained in Theorem \ref{prop:asymp:lowerbound} is $\frac{1}{12}$.
Table \ref{fig:2d:rel:errors} shows the relative error $\left|\sigma^t_{1,2^{k-1}}-\sigma^t_{1,2^k}\right|$ for some annuli in this figure.
\end{example}

\begin{example}\label{example:2}
Figure \ref{2D:variouscases} plots the eigenvalues for the annuli in two dimensions given by $r_2=1$, $r_1=0.2,\, 0.4,\, 0.6,\,0.8$ and $\frac{t}{r_2-r_1}=0,\,0.02,\, 0.04,\, \dots,\, 0.98.$
\end{example}


\begin{figure}[h!]
\centering

\subfloat{
\includegraphics[width=.5\textwidth,  trim=4.1cm 2.5cm 3.8cm 1.8cm, clip]{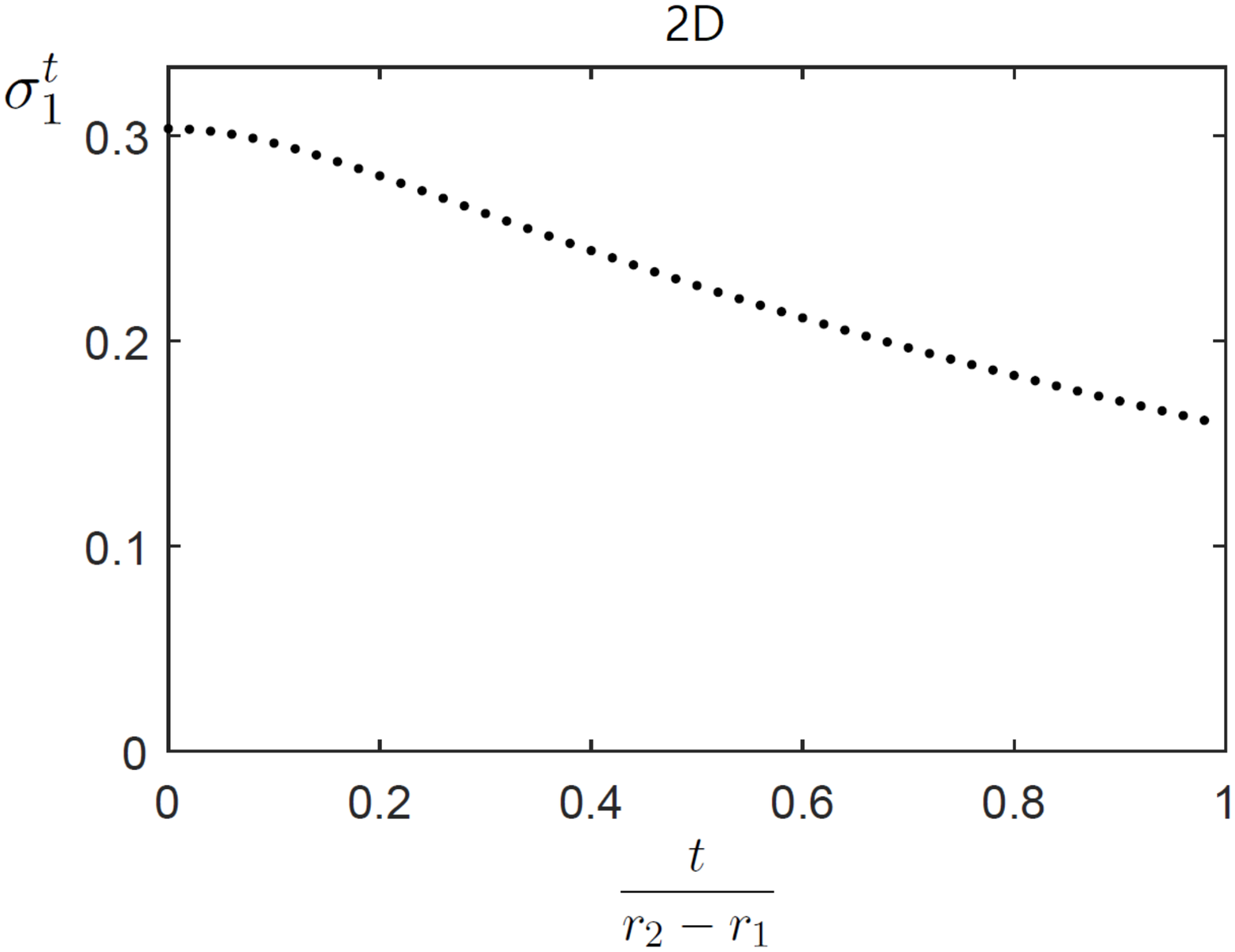}
}\quad
\subfloat{
\begin{tikzpicture}[scale=0.65]

\coordinate  (C) at (-1.5, 0);
\coordinate (X) at (0, 0);
\fill[gray!40,even odd rule] (X) circle (1) (C) circle (3);

\draw (-1.5,0) circle (3);
\draw (0, 0) circle (1);
\fill (-1.5, 0) circle (0.07);
\fill (0, 0) circle (0.07);

\draw[draw=none] (-5.7, 0) -- (2.7, 0);
\draw[draw=none] (0, -5) -- (0, -2);

\draw (0, -0.2) -- (-1.5, -0.2);
\fill (0, -0.2) -- (-0.3, -0.3) -- (-0.3, -0.1);
\fill (-1.5, -0.2) -- (-1.2, -0.3) -- (-1.2, -0.1);

\draw (-0.7, 0.2) node {$t$};

\draw (-1.5, 3.35) node {$B_2$};
\draw (0, 1.35) node {$B_1^t$};
\draw[draw=none] (0,0) -- (0,-4.5);
\end{tikzpicture}
}

\caption{\label{fig:2d:numerics}
The first Steklov--Dirichlet eigenvalue for $B_2\backslash \overline{B_1^t}\subset\mathbb{R}^2$ with $r_1=1,\,r_2=3$ and 
$\frac{t}{r_2-r_1}=0,\,0.02,\dots,\,0.98$ ($50$ cases). 
All the cases except $t=0$ are numerically computed following the stopping criterion \eqnref{relative:stop}; at $t=0$, we plot the exact value $\left(r_2(\ln r_2 - \ln r_1)\right)^{-1}$. The numerical values of all cases comply with the conjecture that $\sigma_1^t$ is monotone decreasing in $t$.\\[5mm]}
\end{figure}
\begin{table}[h!]
\centering
    \begin{minipage}[t]{0.49\textwidth}
\strut\vspace*{-\baselineskip}\newline
        \begin{tabular}{|c| c | c | c |}
           \hline
           $\frac{t}{r_2-r_1}$ & $k$ & $\sigma_{1,2^k}^t$ & $\sigma_{1,2^{k-1}}^t-\sigma_{1,2^k}^t$ \\ \hline\hline
           0.2 & 3 & 0.280415816567 & \\ 
            & 4 & 0.280415816560 & 7.32098E-12 \\
		 & 5 & 0.280415816559 & 2.67508E-13
\\ \hline\hline
           0.4 & 3 & 0.243981453018 & \\ 
            & 4 & 0.243981314075 & 1.38943E-07 \\ 
            & 5 & 0.243981314075 & 1.20820E-13
\\ \hline\hline
		0.6 & 3 & 0.211232489807 & \\ 
		 & 4 & 0.211194760285 & 3.77295E-05 \\ 
		 & 5 & 0.211194759856 & 4.29199E-10 \\ 
		 & 6 & 0.211194759856 & 4.91829E-14
\\ \hline
        \end{tabular}
    \end{minipage}\hfill
    \begin{minipage}[t]{0.49\textwidth}\centering
\strut\vspace*{-\baselineskip}\newline
        \begin{tabular}{|c| c | c | c |}
           \hline
           $\frac{t}{r_2-r_1}$ & $k$ & $\sigma_{1,2^k}^t$ & $\sigma_{1,2^{k-1}}^t-\sigma_{1,2^k}^t$ \\ \hline\hline
           0.8 & 3 & 0.185487114250 & \\ 
            & 4 & 0.183172148523 & 2.31497E-03 \\
            & 5 & 0.183167795557 & 4.35297E-06 \\
            & 6 & 0.183167795551 & 6.58637E-12 \\
            & 7 & 0.183167795551 & 4.10783E-15
\\ \hline\hline
           0.98 & 3 & 0.359778077514 & \\ 
            & 4 & 0.191032748174 & 1.68745E-01 \\ 
            & 5 & 0.162471199179 & 2.85615E-02 \\
		 & 6 & 0.161289731970 & 1.18147E-03 \\
		 & 7 & 0.161288441910 & 1.29006E-06 \\
		 & 8 & 0.161288441909 & 7.36411E-13
\\ \hline
        \end{tabular}
    \end{minipage}
\vskip 3mm
\caption{\label{fig:2d:rel:errors} 
Relative errors for some annuli considered in Figure \ref{fig:2d:numerics}. 
As $\p B_1^t$ gets closer to $\p B_2$ (i.e., $t$ increases), we need a bigger truncated matrix to satisfy the stopping criterion \eqnref{relative:stop}.
}
\end{table}
\clearpage

\begin{figure}[h!]
\centering
\includegraphics[width=0.7\textwidth, trim=.5cm 3.5cm .5cm 3cm,  clip]{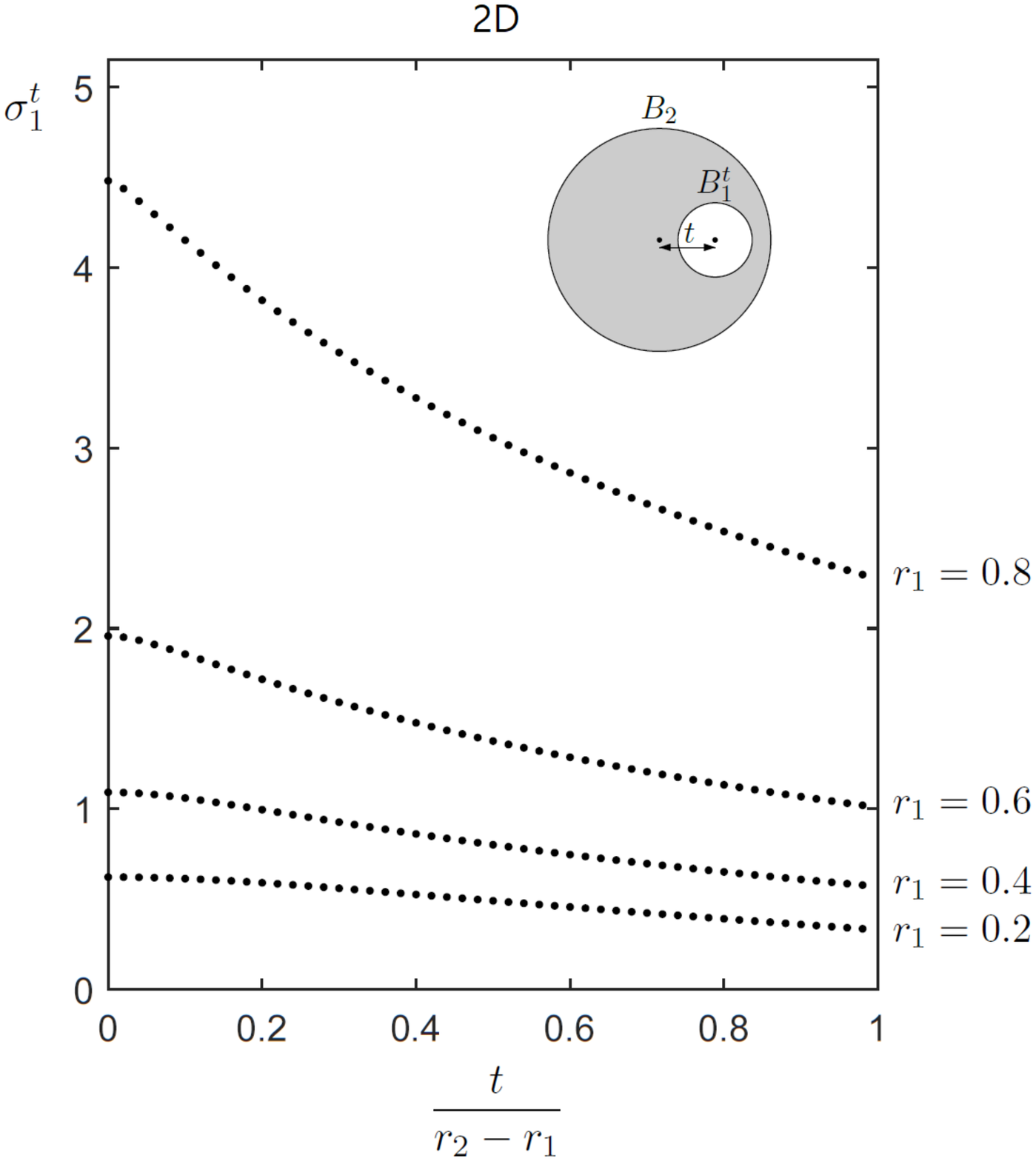}
\caption{
Numerical values of $\sigma_1^t$ for various possible values of $r_1$ and $t$ in two dimensions, where $r_2$ is fixed to be $1$. The numerical values of all cases comply with the conjecture that $\sigma_1^t$ is monotone decreasing in $t$.
}\label{2D:variouscases}
\end{figure}

\subsection{Proofs of Lemma \ref{lemma:num:eig:decrease} and Proposition \ref{p:num:2d}}\label{sec:Proof:numer}

\noindent{\textbf{Proof of Lemma \ref{lemma:num:eig:decrease}}}\hskip .3cm
One show inductively that 
\begin{align}\label{det_M_n}
\mbox{det}\left(M_n\right)&=\frac{1}{\ds \alpha^n} {\prod_{k=0}^{n-1}d_k^2}
\begin{vmatrix}
\cosh\xi_2 & 1 & &&0\\
1&2\cosh\xi_2&1&&\\
&1&2\cosh\xi_2&\ddots&\\
&&\ddots&\ddots&1\\
0&&&1&2\cosh\xi_2
\end{vmatrix}
=\frac{1}{\alpha^n}\left({\prod_{k=0}^{n-1}d_k^2}\right) {\cosh(n\xi_2)}>0.
\end{align}
All of the submatrices of $M_n$ are also $M_k$ for some $k$ and, hence, have positive determinant values. Therefore, $M_n$ is positive definite and $\sigma_{1,n}^t>0$ for all $n\in\mathbb{N}$.
In the remaining of this proof, we show $\sigma_{1,n+1}^t<\sigma_{1,n}^t$ by induction on $n$.

Set 
$$p_n(\lambda):=\det(M_n-\lambda I).$$
We note that $\sigma_{1,n}^t$ is the smallest positive solution to $p_n(\lambda)=0$. 
From the fact that $p_n(0)=\mbox{det}(M_n)>0$ (see \eqnref{det_M_n}) and the intermediate value theorem, it holds that for each $n$,
\begin{align} \label{p_n:cond}
&p_n(\sigma_{1,n}^t)=0,\\ \label{p_n:lambda:general}
&p_n(\lambda)>0\quad\mbox{for all }0<\lambda<\sigma_{1,n}^t.
\end{align}

In view of \eqnref{M_n:full}, one can obtain the recursive relation:
\begin{align*}
p_2(\lambda)&=\left(\frac{1}{\alpha}{2\cosh\xi_2}\cdot{d_1^2}-\lambda\right)p_1(\lambda)- \frac{1}{\alpha^2}\left(d_0d_1\right)^2,\\ 
\ds p_{n+2}(\lambda)&=\left(\frac{1}{\alpha}{2\cosh\xi_2}\cdot {d_{n+1}^2}-\lambda\right)p_{n+1}(\lambda)-\frac{1}{\alpha^2}\left(d_{n+1}d_{n}\right)^2 p_{n}(\lambda)\quad\mbox{for }n\geq1.
\end{align*}
It then holds from \eqnref{p_n:cond} that
\begin{align}\label{eq:p2:ineq}
p_2\left(\sigma_{1,1}^t\right)&=- \frac{1}{\alpha^2}\left(d_0d_1\right)^2<0,\\
\label{eq:qn:subst}
p_{n+2}\left(\sigma_{1,n+1}^t\right)&=-\frac{1}{\alpha^2}\left(d_{n+1}d_{n}\right)^2 p_{n}\left(\sigma_{1,n+1}^t\right)\quad\mbox{for }n\geq1.
\end{align}
 From \eqnref{p_n:cond}, \eqnref{p_n:lambda:general} and \eqnref{eq:p2:ineq}, we have $\sigma_{1,2}^t<\sigma_{1,1}^t.$
Now, we suppose that $\sigma_{1,n+1}^t<\sigma_{1,n}^t$ for some $n\ge1$, then it follows from \eqnref{p_n:lambda:general} that $p_{n}(\sigma_{1,n+1}^t)>0$ and, hence, $p_{n+2}(\sigma^t_{1,n+1})<0$ because of \eqnref{eq:qn:subst}. From \eqnref{p_n:cond} and \eqnref{p_n:lambda:general}, we have $\sigma_{1,n+2}^t<\sigma_{1,n+1}^t$. By induction, we complete the proof.
\qed

\noindent{\textbf{Proof of Proposition \ref{p:num:2d}}}\hskip .3cm 
The right inequality is a direct consequence of the variational characterization \eqnref{variational characterization2} because of the fact that $u_{1,m}^t \in H^1_{\bar{B}_1^t}(B_2)\setminus\{0\}$ for each $m$. In the following we prove the left inequality. 

We remind the reader that $P_n L P_n$ is a positive definite symmetric matrix on $H_n$ with respect to the inner product $(\cdot,\cdot)$ defined by \eqnref{H:inner}, which is identical to the finite dimensional positive definite matrix \eqnref{M_n:full}. Hence, the first eigenvalue $\sigma_{1,n}^t$ of $P_n L P_n$ also admits a variational characterization similar to \eqnref{variational characterization2}:
\begin{align}\notag\ds
\sigma_{1,n}^t&=\inf\left\{\frac{\left(P_n L P_n v,\,v\right)}{(v,\,v)}\,:\,v\in H_n\setminus\{0\}\right\}.
\end{align}
By taking $v=P_n u_1^t$, we obtain
\begin{align*}
\sigma^t_{1,n}&\le\frac{\left(P_n L P_n u_{1}^t,\, P_nu_{1}^t\right)}{(P_nu_{1}^t,\,P_nu_{1}^t)},\\
\notag\ds\sigma_{1,n}^t-\sigma_1^t&\le \frac{\left(P_n L P_n u_{1}^t,\, P_nu_{1}^t\right)}{(P_nu_{1}^t,\,P_nu_{1}^t)}- \sigma_1^t  \frac{\left( P_n u_{1}^t,\, P_nu_{1}^t\right)}{(P_nu_{1}^t,\,P_nu_{1}^t)}
=\frac{\left(P_n L [P_n u_{1}^t- u_{1}^t],\, P_nu_{1}^t\right)}{\left(P_nu_{1}^t,\,P_nu_{1}^t\right)}.
\end{align*}
Now, by using the series expression \eqnref{seriesfirst} of $u_1^t$, we obtain
\begin{align*}
&\left(P_n L [P_n u_{1}^t- u_{1}^t],\, P_nu_{1}^t\right)\\
&=-\big(P_n L[\cos(n\theta)],\,\cos((n-1)\theta)\big)\frac{4}{n(n-1)}A_n A_{n-1}\sinh(n(\xi_1-\xi_2))\sinh((n-1)(\xi_1-\xi_2)),\\
&\left(P_nu_{1}^t,\,P_nu_{1}^t\right)
={a_0^2}{d_0^2}\left(1-\frac{\xi_2}{\xi_1}\right)^2+\sum_{k=1}^{n-1}\frac{4}{k^2}{d_k^2}A_k^2\sinh^2(k(\xi_1-\xi_2))\geq{a_0^2}{d_0^2}\left(1-\frac{\xi_2}{\xi_1}\right)^2.
\end{align*}
Hence, it follows that
\begin{align}\label{error:sigma:last}
\sigma_{1,n}^t-\sigma_1^t\ds&\le \frac{\ds \frac{1}{\alpha}\frac{4}{n(n-1)}d_n^2 d_{n-1}^2 \left|A_n A_{n-1}\right|\sinh(n(\xi_1-\xi_2))\sinh((n-1)(\xi_1-\xi_2))}{\ds{a_0^2}{d_0^2}\left(1-\frac{\xi_2}{\xi_1}\right)^2}.
\end{align}
We note that the term $\frac{4}{n(n-1)}d_n^2 d_{n-1}^2$ is uniformly bounded independently of $n$.

As $u_1^t$ satisfies the Robin boundary condition with constant ratio $\sigma_1^t$ on the circular boundary $\p B_2$, one can analytically extend $u_1^t$ across $\p B_2$. Therefore, the right-hand side of \eqnref{seriesfirst} converges at $\xi$, satisfying $\widetilde{\xi}_2\leq\xi\leq \xi_1$ for some $\widetilde{\xi}_2<\xi_2$. 
In particular, at $({\xi},\theta)=(\widetilde{\xi_2},\pi)$, it holds that 
$$\left|A_n\right|\leq Cne^{-n(\xi_1-\widetilde{\xi_2})}$$
for some constant $C$. Hence,  for sufficiently large $n$, we then have
\beq\label{AnAnminus1}\left|A_n A_{n-1}\right|\sinh(n(\xi_1-\xi_2))\sinh((n-1)(\xi_1-\xi_2))\le  e^{-n(\xi_2-\widetilde{\xi}_2)}.\eeq
Hence, the right-hand side of \eqnref{error:sigma:last} tends to $0$ as $n\to\infty$.
Since the sequence $\{\sigma_{1,n}^t\}$ is convergent thanks to Lemma \ref{lemma:num:eig:decrease}, we conclude that
$\lim_{n\to\infty}\sigma_{1,n}^t-\sigma_1^t\le 0.$
This completes the proof.
\qed

\section{Conclusion}
We investigated the first Steklov--Dirichlet eigenvalue of a domain bounded by two balls of given radii and its monotonicity with respect to the distance between the two centers.
We proved the differentiability of the eigenvalue and obtained an integral expression for the derivative value.
For the planar annulus case, we estimated the ratio of consecutive coefficients, $F_n$, in the series expansion of the first eigenvalue in bipolar coordinates. As an application of this estimate, we derived an explicit lower bound of the first eigenvalue given that two circular boundaries of the annulus are sufficiently close. 
We performed numerical computations in two dimensions to numerically verify the monotonicity of the first eigenvalue. The estimate on $F_n$ may lead to an analytical proof for the shape monotonicity of first Steklov--Dirichlet eigenvalue on an eccentric annulus. The monotonicity is about comparing the two first eigenvalues for eccentric annuli, so it might be necessary to measure the asymmetry of the domains to prove the property.



\begin{thebibliography}{10}

\bibitem{Agranovich:2006:MPS}
Mikhail~Semenovich Agranovich, \emph{On a mixed {P}oincar\'{e}-{S}teklov type
  spectral problem in a {L}ipschitz domain}, Russ. J. Math. Phys. \textbf{13}
  (2006), 239--244.

\bibitem{Aithal:2005:TFL}
A.~R. Aithal and M.~H.~C. Anisa, \emph{On two functionals connected to the
  {L}aplacian in a class of doubly connected domains in space-forms}, Proc.
  Indian Acad. Sci. Math. Sci. \textbf{115} (2005), no.~1, 93--102.

\bibitem{Ammari:2005:GES}
Habib Ammari, Hyeonbae Kang, and Mikyoung Lim, \emph{Gradient estimates for
  solutions to the conductivity problem}, Math. Ann. \textbf{332} (2005),
  no.~2, 277--286.

\bibitem{Anisa:2015:EOP}
M.~H.~C. Anisa and Rajesh Mahadevan, \emph{An eigenvalue optimization problem
  for the {$p$}-{L}aplacian}, Proc. Roy. Soc. Edinburgh Sect. A \textbf{145}
  (2015), no.~6, 1145--1151.

\bibitem{Anisa:2013:TFL}
M.~H.~C. Anisa and Murali~Krishna Vemuri, \emph{Two functionals connected to
  the {L}aplacian in a class of doubly connected domains on rank one symmetric
  spaces of non-compact type}, Geom. Dedicata \textbf{167} (2013), 11--21.

\bibitem{Anoop:2018:SMF}
Thazle~Veetle Anoop, Vladimir Bobkov, and Sarath Sasi, \emph{On the strict
  monotonicity of the first eigenvalue of the {$p$}-{L}aplacian on annuli},
  Trans. Amer. Math. Soc. \textbf{370} (2018), no.~10, 7181--7199.

\bibitem{Arrieta:2008:FRL}
Jos\'{e}~M. Arrieta, \'{A}ngela Jim\'{e}nez-Casas, and An\'{\i}bal
  Rodr\'{\i}guez-Bernal, \emph{Flux terms and {R}obin boundary conditions as
  limit of reactions and potentials concentrating at the boundary}, Rev. Mat.
  Iberoam. \textbf{24} (2008), no.~1, 183--211.

\bibitem{Banuelos:2010:EIM}
Rodrigo Ba\~{n}uelos, Tadeusz Kulczycki, Iosif Polterovich, and Bartlomiej
  Siudeja, \emph{Eigenvalue inequalities for mixed {S}teklov problems},
  Operator theory and its applications, Amer. Math. Soc. Transl. Ser. 2, vol.
  231, Amer. Math. Soc., Providence, RI, 2010, pp.~19--34.

\bibitem{Bandle:1980:IIA}
Catherine Bandle, \emph{Isoperimetric inequalities and applications},
  Monographs and Studies in Mathematics, vol.~7, Pitman (Advanced Publishing
  Program), Boston, Mass.-London, 1980.

\bibitem{Brezis:2011:FAS}
Haim Brezis, \emph{Functional analysis, {S}obolev spaces and partial
  differential equations}, Universitext, Springer, New York, 2011.

\bibitem{Dambrine:2016:EEP}
Marc Dambrine, Mohammed El-Djalil Kateb, and Jimmy Lamboley, \emph{An extremal
  eigenvalue problem for the {W}entzell-{L}aplace operator}, Ann. Inst. H.
  Poincar\'{e} Anal. Non Lin\'{e}aire \textbf{33} (2016), no.~2, 409--450.

\bibitem{Dittmar:1998:IIS}
Bodo Dittmar, \emph{Isoperimetric inequalities for the sums of reciprocal
  eigenvalues}, Progress in partial differential equations, {V}ol. 1
  ({P}ont-\`a-{M}ousson, 1997), Pitman Res. Notes Math. Ser., vol. 383,
  Longman, Harlow, 1998, pp.~78--87.

\bibitem{Dittmar:2005:EPC}
\bysame, \emph{Eigenvalue problems and conformal mapping}, Handbook of complex
  analysis: geometric function theory. {V}ol. 2, Elsevier Sci. B. V.,
  Amsterdam, 2005, pp.~669--686.

\bibitem{Dittmar:2000:ZKE}
Bodo Dittmar and Reiner K\"{u}hnau, \emph{Zur {K}onstruktion der
  {E}igenfunktionen {S}tekloffscher {E}igenwertaufgaben}, Z. Angew. Math. Phys.
  \textbf{51} (2000), no.~5, 806--819.

\bibitem{Dittmar:2003:MSE}
Bodo Dittmar and Alexander~Yu. Solynin, \emph{The mixed {S}teklov eigenvalue
  problem and new extremal properties of the {G}r\"{o}tzsch ring}, Zap. Nauchn.
  Sem. S.-Peterburg. Otdel. Mat. Inst. Steklov. (POMI) \textbf{270} (2000),
  no.~Issled. po Line\u{\i}n. Oper. i Teor. Funkts. 28, 51--79, 365.

\bibitem{Fan:2015:EPS}
Xu-Qian Fan, Luen-Fai Tam, and Chengjie Yu, \emph{Extremal problems for
  {S}teklov eigenvalues on annuli}, Calc. Var. Partial Differential Equations
  \textbf{54} (2015), no.~1, 1043--1059.

\bibitem{Bonder:2007:OFS}
Juli\'{a}n Fern\'{a}ndez~Bonder, Pablo Groisman, and Julio~Daniel Rossi,
  \emph{Optimization of the first {S}teklov eigenvalue in domains with holes: a
  shape derivative approach}, Ann. Mat. Pura Appl. (4) \textbf{186} (2007),
  no.~2, 341--358.

\bibitem{Fraser:2020:ECS}
Ailana Fraser and Pam Sargent, \emph{Existence and classification of
  {$\mathbb{S}^1-$}invariant free boundary minimal annuli and m\"{o}bius bands
  in {$\mathbb{B}^n$}}, J Geom Anal (2020).

\bibitem{Fraser:2013:MSE}
Ailana Fraser and Richard Schoen, \emph{Minimal surfaces and eigenvalue
  problems}, Geometric analysis, mathematical relativity, and nonlinear partial
  differential equations, Contemp. Math., vol. 599, Amer. Math. Soc.,
  Providence, RI, 2013, pp.~105--121.

\bibitem{Fraser:2016:SEB}
\bysame, \emph{Sharp eigenvalue bounds and minimal surfaces in the ball},
  Invent. Math. \textbf{203} (2016), no.~3, 823--890.

\bibitem{Ftouhi:2019:WPS}
Ilias Ftouhi, \emph{Where to place a spherical obstacle so as to maximize the
  first {S}teklov eigenvalue}, hal-02334941v2.

\bibitem{Grebenkov:2013:GSL}
Denis~S. Grebenkov and Binh~T. Nguyen, \emph{Geometrical structure of
  {L}aplacian eigenfunctions}, SIAM Rev. \textbf{55} (2013), no.~4, 601--667.

\bibitem{Hersch:1968:EPI}
Joseph Hersch and Lawrence~E. Payne, \emph{Extremal principles and
  isoperimetric inequalities for some mixed problems of {S}tekloff's type}, Z.
  Angew. Math. Phys. \textbf{19} (1968), 802--817.

\bibitem{Kang:2014:CEF}
Hyeonbae Kang, Mikyoung Lim, and KiHyun Yun, \emph{Characterization of the
  electric field concentration between two adjacent spherical perfect
  conductors}, SIAM J. Appl. Math. \textbf{74} (2014), no.~1, 125--146.

\bibitem{Kesavan:2003:TFL}
Srinivasan Kesavan, \emph{On two functionals connected to the {L}aplacian in a
  class of doubly connected domains}, Proc. Roy. Soc. Edinburgh Sect. A
  \textbf{133} (2003), no.~3, 617--624.

\bibitem{Kim:2018:EFC}
Junbeom Kim and Mikyoung Lim, \emph{Electric field concentration in the
  presence of an inclusion with eccentric core-shell geometry}, Math. Ann.
  \textbf{373} (2019), no.~1-2, 517--551.

\bibitem{Kulczycki:2009:HST}
Tadeusz Kulczycki and Nikolay Kuznetsov, \emph{`{H}igh spots' theorems for
  sloshing problems}, Bull. Lond. Math. Soc. \textbf{41} (2009), no.~3,
  495--505.

\bibitem{Kuzentsov:2014:LVA}
Nikolay Kuznetsov, Tadeusz Kulczycki, Mateusz Kwa\'{s}nicki, Alexander Nazarov,
  Sergey Poborchi, Iosif Polterovich, and Bart\l~omiej Siudeja, \emph{The
  legacy of {V}ladimir {A}ndreevich {S}teklov}, Notices Amer. Math. Soc.
  \textbf{61} (2014), no.~1, 9--22.

\bibitem{Labrie:2017:TSS}
Marc-Antoine Labrie, \emph{Le th{\'{e}}or{\`{e}}me spectral pour le
  probl{\`{e}}me de {S}teklov sur un domaine euclidien}, Master's thesis,
  Universit{\'{e}} Laval (2017).

\bibitem{Lamberti:2015:VSE}
Pier~Domenico Lamberti and Luigi Provenzano, \emph{Viewing the {S}teklov
  eigenvalues of the {L}aplace operator as critical {N}eumann eigenvalues},
  Current trends in analysis and its applications, Trends Math.,
  Birkh\"{a}user/Springer, Cham, 2015, pp.~171--178.

\bibitem{Lamberti:2017:NSA}
\bysame, \emph{Neumann to {S}teklov eigenvalues: asymptotic and monotonicity
  results}, Proc. Roy. Soc. Edinburgh Sect. A \textbf{147} (2017), no.~2,
  429--447.

\bibitem{Lim:2015:ASC}
Mikyoung Lim and Sanghyeon Yu, \emph{Asymptotics of the solution to the
  conductivity equation in the presence of adjacent circular inclusions with
  finite conductivities}, J. Math. Anal. Appl. \textbf{421} (2015), no.~1,
  131--156.

\bibitem{Matthiesen:2020:FBM}
Henrik Matthiesen and Romain Petrides, \emph{Free boundary minimal surfaces of
  any topological type in euclidean balls via shape optimization},
  arXiv:2004.06051.

\bibitem{Necas:2012:DMT}
Jind\v{r}ich. Ne\v{c}as, \emph{Direct methods in the theory of elliptic
  equations}, Springer Monographs in Mathematics, Springer, Heidelberg, 2012.

\bibitem{Paoli:2020:SRS}
Gloria Paoli, Gianpaolo Piscitelli, and Rossano Sannipoli, \emph{A stability
  result for the {S}teklov {L}aplacian eigenvalue problem with a spherical
  obstacle}, arXiv:1909.06579.

\bibitem{Petrides:2019:MSE}
Romain Petrides, \emph{Maximizing {S}teklov eigenvalues on surfaces}, J.
  Differential Geom. \textbf{113} (2019), no.~1, 95--188.

\bibitem{Ramm:1998:IME}
Alexander~G. Ramm and Pappur~N. Shivakumar, \emph{Inequalities for the minimal
  eigenvalue of the {L}aplacian in an annulus}, Math. Inequal. Appl. \textbf{1}
  (1998), no.~4, 559--563.

\bibitem{aithal:2019:FDE}
Akanksha~V. Rane and A.~R. Aithal, \emph{The first {D}irichlet eigenvalue of
  the {L}aplacian in a class of doubly connected domains in complex projective
  space}, Indian J. Pure Appl. Math. \textbf{50} (2019), no.~1, 69--81.

\bibitem{Quinones:2019:CDF}
Leoncio Rodriguez-Qui\~{n}ones, \emph{A critical domain for the first
  normalized nontrivial {S}teklov eigenvalue among planar annular domains},
  arXiv:1909.02121.

\bibitem{Verma:2018:EPL}
Gopalakrishnan Santhanam and Sheela Verma, \emph{On eigenvalue problems related
  to the {L}aplacian in a class of doubly connected domains}, arXiv:1803.05750.

\bibitem{SBH:2019:VTE}
Francisco-Javier Sayas, Thomas~S. Brown, and Matthew~E. Hassell,
  \emph{Variational techniques for elliptic partial differential equations},
  CRC Press, Boca Raton, FL, 2019, Theoretical tools and advanced applications.

\bibitem{Seo:2019:SOP}
Dong-Hwi Seo, \emph{A shape optimization problem for the first mixed
  {S}teklov-{D}irichlet eigenvalue}, arXiv:1909.06579.

\bibitem{stekloff:1902:FPM}
W.~Stekloff, \emph{Sur les probl\`emes fondamentaux de la physique
  math\'{e}matique (suite et fin)}, Ann. Sci. \'{E}cole Norm. Sup. (3)
  \textbf{19} (1902), 455--490.

\bibitem{Zeidler:1986:NFA}
Eberhard Zeidler, \emph{Nonlinear functional analysis and its applications.
  {I}}, Springer-Verlag, New York, 1986, Fixed-point theorems, Translated from
  the German by Peter R. Wadsack.

\end{thebibliography}

\providecommand{\bysame}{\leavevmode\hbox to3em{\hrulefill}\thinspace}
\providecommand{\MR}{\relax\ifhmode\unskip\space\fi MR }
\providecommand{\MRhref}[2]{%
  \href{http://www.ams.org/mathscinet-getitem?mr=#1}{#2}
}
\providecommand{\href}[2]{#2}

\end{document}